\documentclass[11pt]{amsart}
\usepackage{amsmath,amscd,amssymb}
\usepackage[mathscr]{eucal}
\usepackage[all]{xy}

\newcommand{\Wip}{\mathrm{A}_+^1}

\newcommand{\vanish}[1]{\relax}

\newcommand{\N}{\mathbb{N}}
\newcommand{\Z}{\mathbb{Z}}

\newcommand{\R}{\mathbb{R}}
\newcommand{\C}{\mathbb{C}}

\newcommand{\D}{\mathbb{D}}

\newcommand{\Sum}[2][\relax]{%
 \ifx#1\relax \sideset{}{_{#2}}\sum
 \else \sideset{}{^{#1}_{#2}}\sum
 \fi}

\newcommand{\abs}[1]{\left| #1 \right|}
\renewcommand{\abs}[1]{\left\vert#1\right\vert}

\DeclareMathOperator{\Sect}{Sect}

\DeclareMathOperator{\dom}{dom}
\DeclareMathOperator{\ran}{ran}

\newcommand{\cls}[1]{\overline{#1}}

\DeclareMathOperator{\Lin}{\mathcal{L}}

\newcommand{\norm}[2][\relax]{%
   \ifx#1\relax \ensuremath{\left\Vert#2\right\Vert}
   \else \ensuremath{\left\Vert#2\right\Vert_{#1}}
   \fi}

\makeatletter
\newcommand{\sprod}[2]{\ensuremath{%
  \setbox0=\hbox{\ensuremath{#2}}
  \dimen@\ht0
  \advance\dimen@ by \dp0
  \left(\left.#1\rule[-\dp0]{0pt}{\dimen@}\,\right|#2\hspace{1pt}\right)}}
 \makeatother

\newcounter{aufzi}

\newcounter{aufzii}

\newcounter{aufziii}

 \newtheorem{thm}{Theorem}[section]
 \newtheorem{cor}[thm]{Corollary}
 \newtheorem{lemma}[thm]{Lemma}
 \newtheorem{prop}[thm]{Proposition}

 \theoremstyle{definition}
 \newtheorem{defn}[thm]{Definition}

 \theoremstyle{remark}
 \newtheorem{rem}[thm]{Remark}

\newtheorem{example}[thm]{Example}

\newtheorem{remark}[thm]{Remark}

\numberwithin{equation}{section}

\numberwithin{equation}{section}
\numberwithin{theorem}{section}

\begin{document}

\title[On discrete subordination]{On discrete subordination of power bounded and Ritt operators}
\author{Alexander Gomilko}
\address{Faculty of Mathematics and Computer Science\\
Nicolas Copernicus University\\
ul. Chopina 12/18\\
87-100 Toru\'n, Poland \\
and Institute of Telecommunications and Global \\
Information Space, National Academy of Sciences of Ukraine\\
Kiev, Ukraine
}

\email{gomilko@mat.umk.pl}

\author{Yuri Tomilov}
\address{Institute of Mathematics\\
Polish Academy of Sciences\\
\'Sniadeckich 8\\
00-956 Warsaw, Poland}
 \email{ytomilov@impan.pl}

\thanks{This work was  partially supported by the NCN grant
 DEC-2014/13/B/ST1/03153 and by the EU grant  ``AOS'', FP7-PEOPLE-2012-IRSES, No 318910.}

\subjclass{Primary 47A60, 47D03; Secondary 46N30, 26A48}

\keywords{Ritt operator, power bounded operator, functional
calculus, subordination, Hausdorff functions, Bernstein functions,
holomorphic $C_0$-semigroup}

\date{\today}

\begin{abstract}
By means of a new technique, we develop further a discrete subordination approach to the
functional calculus of  power bounded and Ritt operators initiated
by N. Dungey in \cite{Dungey}. This allows us to show, in
particular, that  (infinite) convex combinations of powers of Ritt
operators are Ritt. Moreover, we provide a unified framework for
several main results on discrete subordination  from \cite{Dungey}
and answer a question left open in \cite{Dungey}. The paper can be
considered as a complement to \cite{GT} for the discrete setting.

\end{abstract}

\maketitle
\section{Introduction}

The aim of this paper is to initiate a study of permanence and
``improving'' properties of discrete subordination for bounded
operators parallel in a sense to an investigation of subordination
for $C_0$-semigroups realized in our recent paper \cite{GT}. A
discrete subordination in the abstract setting has not received a
proper attention in the literature, and we are not aware of any
relevant works apart from  \cite{Dungey} and \cite{Bendikov}. At
the same time, our
 paper can be regarded as a contribution to understanding of permanence and improving properties for
 functional calculi of bounded operators.
In fact, there are very few results saying that basic features of
operator like resolvent estimates or asymptotics of powers, are
preserved under a functional calculus or, at least, under a
substantial class of admissible functions. It seems, the only
relevant and nontrivial result so far  was the one by Hirsch
\cite{Hirsch} saying that complete Bernstein functions preserve
the class of sectorial (in general, unbounded) operators.

To put our results into a proper context, we first recall several
basic facts stemming from the subordination theory of
$C_0$-semigroups on Banach spaces. There are two basic notions
behind the subordination theory: the notion of Bernstein function
and that of subordinator.

Recall that a family of positive subprobability Borel measures
$(\mu_t)_{t\ge 0}$ on $[0,\infty)$ is said to be said to be a
\emph{subordinator} if for all $s,t\ge0$ one has
$\mu_{t+s}=\mu_t*\mu_s,$ and $\lim_{t\to0+}\mu_t=\delta_0$ in the
$w^*$-topology of the space of bounded Borel measures on $[0,\infty).$
Given a subordinator $(\mu_t)_{t\ge 0}$ one may define a Bernstein
function $\psi$ on $(0,\infty)$ by the formula
\begin{equation}\label{bernstein}
e^{-t\psi(\lambda)}=\int_{0}^{\infty}e^{-s\lambda} \, \mu_t(ds),
\end{equation}
for all $t \ge 0$ and $\lambda >0.$ See \cite[Section
5]{SchilSonVon2010} for more on that. (Alternatively, a positive
and smooth function $\psi$ on $(0, \infty)$ is called a {\em
Bernstein function} if $(-1)^n\frac{d^n \psi(t)}{
d {t}^n}\leq 0 $ for all $n \in \mathbb N$ and  $t> 0.$)

If now  $(e^{-tA})_{t \ge 0}$  is a bounded $C_0$-semigroup on a
(complex) Banach space $X$ with generator $-A,$  and  $(\mu_t)_{t \ge 0}$ is
a subordinator, then, following  intuition provided by
\eqref{bernstein}, one can define a new \emph{bounded}
$C_0$-semigroup on $X$ as
\begin{equation}\label{bersem}
T(t):=\int_{0}^{\infty}e^{-sA}\,\mu_t(ds), \qquad t \ge 0,
\end{equation}
where the (Bochner) integral converges in the strong topology of
$X.$ Once again, in view of \eqref{bernstein}, it is natural to
consider the generator of $(T(t))_{t \ge 0}$ as a (minus)
Bernstein function $\psi$ of $A.$ This appears to be a right
choice and can serve as the definition of $\psi(A)$ indeed. There
are several other alternative definitions of $\psi(A),$ but all of
them lead to the same operator. The operator Bernstein functions
have a number of natural properties resembling that of scalar
functions. One of these properties is expressed by \eqref{bersem}
and provides a natural way  to construct a bounded semigroup
 $(e^{-t\psi(A)})_{t \ge 0}$ by means of a given bounded semigroup $(e^{-tA})_{t \ge 0}$  and
 a subordinator  $(\mu_t)_{t \ge 0}$.
In this situation, $(e^{-t\psi(A)})_{t \ge 0}$ is called
subordinated to $(e^{-t A})_{t \ge 0}$ via a subordinator
$(\mu_t)_{t \ge 0}.$ The construction of subordination described
above goes back to Bochner and Phillips and became a crucial tool
in probability theory and functional analysis (and also in
engineering), see e.g. \cite{SchilSonVon2010} and comments to
Section $13$ there. A classical example of subordination is
provided by the semigroup of fractional powers
$(e^{-tA^\alpha})_{t \ge 0}, \alpha \in (0,1)$ (corresponding to
the Bernstein function $\psi(\lambda)=\lambda^\alpha$). It was studied thoroughly in
the $1960$s by Balakrishnan, Kato and  Yosida. A comprehensive
discussion of subordination for $C_0$-semigroups including many
illustrative examples can be found in \cite[Section
13]{SchilSonVon2010}.

Apart from a number of issues of a purely probabilistic origin,
there are two very natural, operator-theoretical questions in the
study of subordination. Namely, whether subordination preserves
the classes of holomorphic  sectorially bounded holomorphic
$C_0$-semigroups and when it possesses improving properties in the
sense that general bounded $C_0$-semigroups are transformed into
holomorphic semigroups. Motivated by a fundamental paper by
Carasso and Kato \cite{Carasso} (and also its subsequent
developments in \cite{Fujita}, \cite{Fujita1} and \cite{Mirotin})
and answering a problem posed in \cite{Robinson} and \cite{Berg},
we have recently obtained the following result where positive
answers to both questions were provided, see \cite[Theorems 6.8
and 7.9]{GT}.

\begin{thm}\label{GTmain}
(a) \, Let $-A$ be the generator of a bounded $C_0$-semigroup on
$X$ such that $A$ is sectorial of angle $\theta\in [0,\pi/2).$
Then for every Bernstein function $\psi$ the operator $\psi(A)$ is
sectorial of the same angle.

\noindent b)\, Let $\psi$ be  a complete Bernstein function and
let $\gamma \in (0,\pi/2)$ be fixed. The following assertions
are equivalent.
\begin{itemize}
\item [(i)]
One has
$$ \psi(\mathbb C_+) \subset
 \overline{\Sigma}_\gamma.$$
\item [(ii)]  For each Banach space $X$ and each generator $-A$ of a bounded
$C_0$-semigroup on $X,$ the operator $\psi(A)$ is sectorial of
angle $\gamma$.
\end{itemize}
\end{thm}

To create a similar ``discrete'' framework, let now $\mu$ be a
probability measure on $\Z_+:=\mathbb N\cup\{0\},$ in other words,
$\mu(k) \ge 0, k \ge 0,$ and $\sum_{k=0}^{\infty}\mu(k)=1.$ If $T$
is a power bounded operator on $X,$ then setting $\widehat
\mu(T):=\sum_{k=0}^{\infty}\mu(k) T^k,$ and denoting $\mu^n:=\mu*
\dots * \mu$ the $n$th  convolution power of $\mu$,
 so that
$\mu^n \in \ell_1(\Z_+), n \in \mathbb N,$ we have
\begin{equation}\label{discr}
\widehat \mu (T)^n=\sum_{k=0}^{\infty} T^k \mu^n(k)=\int_{\Z_+}
T^k d\mu^n(k),
\end{equation}
where the last equality is purely formal. Thus, there is a clear
analogy with the continuous case, and it is natural to say that
$(\widehat \mu(T)^n)_{n \ge 0}$ (or $\widehat \mu (T)$) is
subordinated to $(T^n)_{n \ge 0}$ (or $T$). However, to simplify
our terminology, we will be considering  just the powers of
$\widehat \mu (T)$ defined by means of the power series $\widehat
\mu(z)=\sum_{k \ge 0}\mu(k) z^k$ called sometimes the generating
function of $\mu$. Note that similarly to the case of bounded
$C_0$-semigroups, if $T$ is power bounded, then the operator
$\widehat \mu(T)$ is power bounded as well. This is precisely the
framework of \cite{Dungey}, and one may then study finer
properties of $\widehat\mu(T)$ in terms of the same properties of
$T$ . The attempt to set up a discrete subordination similar to
the one existing in the setting of $C_0$-semigroups was also made
 in \cite{Bendikov}. However, the assumptions of \cite{Bendikov}
seem to be more restrictive than the ones in \cite{Dungey}.

The paper \cite{Dungey} is devoted mainly to the study of the
improving properties of  measures $\mu,$ or, equivalently, of
their generating functions $\widehat \mu$ in the spirit of
\cite{Carasso}. To discuss the relevant results from \cite{Dungey}
in some more detail,  we have to introduce several
operator-theoretical notions. A bounded linear operator $T$ on a
Banach space $X$ is said to be Ritt if there exists $C\ge 1$ such
that
$$\sigma(T)\subset \overline {\D} \qquad
\text{and}
 \qquad \|(\lambda - T)^{-1}\|\le \frac{C}{|\lambda -
1|},\quad |\lambda| >1,$$
where $\overline \D$ stands for closure of the open unit disc $\D.$
The last two conditions can be equivalently
rewritten in the following, formally stronger, form: There exists $\omega \in [0,\pi/2)$ such that
$$\sigma(T)\subset \D \cup\{1\} \qquad \text{and}
 \qquad \|(\lambda - T)^{-1}\|\le \frac{C_{\omega'}}{|\lambda - 1|}, \qquad \lambda \in \mathbb C
 \setminus \left(1-\overline{\Sigma}_{\omega'} \right), $$
for every $\omega'\in (\omega,\pi)$ and an appropriate
$C_{\omega'}\ge 1,$ where $\Sigma_\omega=\{\lambda \in \mathbb C:
|\arg \lambda|<\omega\}$ and
$\Sigma_0=(0,\infty).$ We say that $T$ is \emph{of angle $\omega$}
in this case. (In fact, one can put $\omega=\arccos (1/C)$ here,
\cite{Lyubich}.) Moreover, as we prove in Proposition
\ref{rittchar} below, $T$ is Ritt if and only if there exists $\sigma \ge 1$ such that for every
$\sigma'>\sigma $ and some $C_{\sigma'} \ge 1$ one has
$$
\|(\lambda - T)^{-1}\|\le \frac{C_{\sigma'}}{|\lambda - 1|},
\qquad \lambda \in \C \setminus S_{\sigma'},
$$
where $S_\sigma:=\{z\in \D:\,|1-z|/(1-|z|)< \sigma \}\cup\{1\},$
is a \emph{Stolz domain}. In this situation $T$ is said to be
\emph{of Stolz type} $\sigma.$ See Section \ref{secritt} for a
thorough discussion of these and related notions. There is a
substantial literature devoted to Ritt operators and their various
properties ranging from the role in functional calculi to
applications in ergodic and probability theories. A sample of it
could include \cite{Merdy1}, \cite{Bakaev}, \cite{Blunck1},
\cite{Blunck}, \cite{Cohen}, \cite{CoSa}--\cite{ElFaRa02},
 \cite{Haa2006}, \cite{KMOT}, \cite{Kalton},
\cite{Komatsu}--\cite{Lyubich}, \cite{NaZe}--\cite{Portal},
 \cite{Vitse}, and
\cite{Vitse1}. (Unfortunately, while the topic is vast, there is
no survey on Ritt operators yet.) We only note one more
characterization of Ritt operators saying that $T$ is Ritt if and
only if $T$ is \emph{power bounded}, i.e. $\sup_{n \ge 0}
\|T^n\|<\infty,$ and $\sup_{n \ge 0} n \|T^n-T^{n+1}\|
<\infty.$ In fact, Ritt operators can serve as a discrete analogue
of generators of (sectorially) bounded holomorphic
$C_0$-semigroups, while power bounded operators correspond to
generators of bounded $C_0$-semigroups. See e.g. \cite{Dungey},
\cite{Blunck}, \cite{Blunck1} or \cite{KMOT} and references
therein for comments on that issue.

In analogy with the case of $C_0$-semigroups considered in
\cite{GT}, one can say that a measure $\mu$ is \emph{improving} if
for any power bounded operator $T$ on $X$ the operator $\widehat
\mu (T)$ is Ritt. By means of a general sufficient condition
involving the boundary behavior of the generating function
$\widehat \mu$ in $\D,$
 Dungey proved
in \cite{Dungey} the improving property for several interesting
and important probabilities $\mu.$

In this paper, we will put Dungey's results from \cite{Dungey} in
a broader context and improve several of them. More generally, in
view of the discussion above, it is natural to to try to obtain a
discrete counterpart of Theorem \ref{GTmain} once the notion of a
discrete subordination is adapted. One of the aims of this paper
is to prove the results on permanence and improving properties for
discrete  subordination similar in a sense to Theorem
\ref{GTmain}.

The following statement is one of the main results of this paper.
Recall that a closed operator $A$ on $X$ is sectorial with angle
of sectoriality $\alpha \in [0,\pi)$ if $\sigma(A)\subset
\overline{\Sigma}_\alpha$ and for every $\omega\in (\alpha,\pi)$
there exists $C_\omega>0$ such that
\[
\|\lambda(\lambda-A)^{-1}\|\le C_\omega, \qquad \lambda\not\in
\overline{\Sigma}_\omega.
\]
\begin{thm}\label{main}
Let
\begin{equation}\label{convex}
g(\lambda):=\sum_{n=0}^\infty c_n \lambda^n,\qquad c_n\ge 0,\quad
\sum_{n=0}^\infty c_n=1.
\end{equation}
Then for any Ritt operator  $T$ of Stolz type $\sigma$ on a Banach
space $X,$ the operator $g(T)$ is Ritt and of the same Stolz
type. Moreover, $g(T)$ is of angle $\omega,$ where $\omega$ is a
sectoriality angle of the Cayley transform ${\mathcal C}(T)$ of
$T.$
\end{thm}

This is a discrete counterpart of Theorem \ref{GTmain}, a).
However, the result seems to be stronger than Theorem
\ref{GTmain}, a) (up to a change of frameworks from the continuous
to the discrete one) since no further assumptions are imposed on
the sequence $(c_n)_{n\ge 0}.$

 Using the terminology from \cite{FoWe},
one may call a power series satisfying \eqref{convex}
\emph{convex} and formulate a (part of) statement above by saying
that \emph{a convex power series of a Ritt operator is Ritt}. This
terminology will be used throughout the paper occasionally.

Thus Theorem \ref{main} can be considered as a full discrete
analogue of Theorem \ref{GTmain}, a). However additional specific
features are present here. While angles of Ritt operators are not,
in general, preserved under discrete subordination, we have a good
control over them via the Cayley transform. On the other hand,
discrete subordination preserves  Stolz type of Ritt operators,
and probably it is Stolz type that is an adequate substitute of
sectoriality angle for sectorial operators in the discrete
setting.

Moreover, in this paper, we show that several results of Dungey on improving properties
can in fact be derived,  more or less directly, from
Theorem \ref{GTmain}, b). This is done by  transferring Theorem \ref{GTmain}, b) to the discrete setting.
  Moreover, our approach allows us to answer a question left
open by Dungey and to provide new interesting examples of improving $\mu$.

Turning to improving properties,  we establish an analogue of
Theorem \ref{GTmain}, b) by replacing complete Bernstein functions
with Hausdorff functions. This allows us not only to prove
alternative proofs for several results from \cite{Dungey} (e.g.
Theorems $1.1, 1.2, 1.3$ and Corollary $1.4$ there)  but also to
answer a problem posed in \cite[p. 1735]{Dungey} concerning the
improving property  of the function
$f_\epsilon(\lambda)=1-\frac{1}{\epsilon}\int_{0}^{\epsilon}(1-\lambda)^\alpha\,
d\alpha,\epsilon \in (0,1).$  Moreover, our approach allows us to
equip the results from \cite{Dungey} with additional geometric
properties which are not available via the techniques from
\cite{Dungey}.

To formulate our second main result, we need a notion of a regular
Hausdorff function. We say that  $h(\lambda)=c_0+\sum_{n=1}^\infty
c_n\lambda^n, \lambda \in \D,$ is a  regular Hausdorff function if $c_0\ge 0$
and there exists a bounded positive Borel measure $\nu$ on $[0,1)$
such that
\begin{equation*}\label{HaFF}
c_n=\int_{[0,1)} t^{n-1} \, \nu(dt), \quad n \ge 1,\quad \text{and}
\quad c_0+\int_{[0,1)}\frac{\nu(dt)}{1-t}=1.
\end{equation*}
Note that if $h$ is as above, then all its Taylor coefficients $c_n, n\ge 0,$ are positive and $\sum_{n=0}^{\infty}c_n=1.$
Thus, $h$ is a generating function of a probability on $\Z_+$ given by $(c_n)_{n \ge 0}.$

\begin{thm}\label{main1}
Let $h$ be a non-constant regular Hausdorff function, and let
 $\gamma\in (0,\pi/2)$ be fixed. The following statements are
 equivalent.
 \begin{itemize}
\item [(i)] One has
\begin{equation*}
1-h(\lambda)\subset \overline{\Sigma}_\gamma,\qquad \lambda\in \D.
\end{equation*}
\item [(ii)] For every  Banach space $X$ and every power
bounded operator $T$ on $X$  the operator $h(T)$ is Ritt of angle
$\gamma$.
\end{itemize}
\end{thm}

Let us comment briefly on our techniques and methodology. It is a
notable  feature of the paper that dealing with bounded operators
we apply methods worked out, first of all, to treat unbounded
operators. We remark  that the techniques of the present paper is
quite different from the techniques from \cite{GT}. While
\cite{GT} put an emphasis on intricate functional calculi
arguments, the approach presented here is more direct. It relies
on deriving a suitable resolvent representation for a (rotated)
Nevanlinna-Pick function of the Cayley transform of $T$ and
function-theoretical estimates for certain  Nevanlinna-Pick
functions. This allows us to obtain results which seem to be more
informative than the corresponding statements in \cite{GT} (when
changing the notions appropriately). As an alternative to
functional calculus technique from \cite{GT}, a similar direct
approach to the study of subordination of $C_0$-semigroups was
recently developed in \cite{BGT}.  Note finally that the main results of this paper
have recently found interesting applications to the study of convolution operators on  $\ell^1(\mathbb Z)$,
 see \cite{CCuny} for more details.

\section{Notation}

For a closed linear operator $A$ on a complex Banach space $X$ we
denote by  $\ran(A),$ $\ker(A)$ and $\sigma(A)$ the {\em range},
the {\em kernel} and the {\em spectrum} of $A$, respectively. The
space of bounded linear operators on $X$ is denoted by
$\mathcal{L}(X).$

The closure of a set $S$ will be denoted by
$\overline{S},$ and $f \circ g$ will stand for a  composition of
functions $f$ and $g.$ For any sets $A$ and $B$ from the complex
plane $\C,$ we denote $A+B:=\{a+b: a \in A, b \in B\}$ and
sometimes write $a$ instead of $\{a\}.$

Finally, we let
\[
\mathbb C_{+}=\{\lambda \in \C:\,{\rm Re}\,\lambda>0\},\qquad \mathbb C^{+}=\{\lambda \in \C:\,{\rm Im}\,\lambda>0\}, \qquad
\mathbb Z_+=\mathbb N \cup \{0\},
\]
and  denote
\[
\Sigma_{0}:=(0,\infty), \qquad  \Sigma_{\beta}:=\{\lambda \in \C:\,\, |\arg \lambda|<\beta\}, \, \beta\in
(0,\pi],
\]
and $\D:=\{\lambda \in \mathbb C: |\lambda|<1 \}.$
\section{Function theory}

\subsection{Nevanlinna-Pick and Cayley functions and their mapping properties}\label{funsec}

To develop the functional calculi machinery, we have to introduce
several function classes and describe their basic properties.

Recall that a function $F$ holomorphic in  the upper half-plane
$\mathbb C^+$ is called Nevanlinna-Pick if $F(\mathbb C^+)\subset
\overline{\mathbb C}_+$. Since we will be interested in functions
defined in the right half-plane $\C_+$, we will need a class of
rotated Nevanlinna-Pick functions, namely the class of functions
holomorphic in $\mathbb C_+$ and  mapping $\C_+$ into
$\overline{\mathbb C}_+$. Moreover, the functions from this latter
class that are positive on $(0,\infty)$ will play a special role.
Thus, eventually, we will work with the class of functions $F$
denoted by $\mathcal{NP_+}$ and described as
$$
\mathcal{NP}_+:=\{F \,\, \text{is holomorphic in}\,\,  \C_+:
F(\mathbb C_+) \subset \overline{\mathbb C}_+ \, \, \text{and}
\,\, F((0,\infty))\subset [0,\infty) \}
$$
Note that the symmetry principle implies $F(\lambda)= \overline{ F(\bar \lambda)}, \lambda
\in \C_+.$

Recall that since the function $ f(\lambda):=iF(-i\lambda), \lambda\in \C^{+}, $ is
Nevanlinna-Pick, the well-known Herglotz theorem (see e. g.
\cite[Corollary 6.8]{SchilSonVon2010}) implies that $f: (0,\infty) \to [0,\infty)$ and
\begin{equation}\label{herg}
f(\lambda)=iF(-i\lambda)=\alpha +a \lambda+ \int_{-\infty}^\infty
\frac{1+\lambda t}{t-\lambda}\,\rho(dt), \qquad \lambda \in \mathbb C^+,
\end{equation}
where  $\alpha\in \R$, $a \ge 0$, and $\rho$ is a positive finite
Borel measure on the real line.

The following theorem contains several properties of $\mathcal{NP_+}$-functions crucial for the sequel.
\begin{thm}\label{NP}
Let $F \in \mathcal{NP_+}.$ Then the following statements hold.
\begin{itemize}
\item [(i)] One has
\begin{equation}\label{A}
F(\lambda)=a \lambda+\frac{b}{\lambda}+2\lambda\int_{(0,\infty)}
\frac{(1+t^2)\,\rho(dt)}{\lambda^2+t^2},\qquad \lambda\in \C_{+},
\end{equation}
where  $a\ge 0$, $b\ge 0,$ and $\rho$  is a positive finite
Borel measure on $(0,\infty)$.
\item [(ii)] For every $\omega \in [0,\pi/2)$ there exists $c_\omega>0$ such that
\begin{equation}\label{sbounds}
|F(\lambda)|\le c_\omega \left(|\lambda|+|\lambda|^{-1}\right), \qquad \lambda \in\Sigma_\omega.
\end{equation}
\item [(iii)] For every $\omega\in [0,\pi/2),$
\begin{equation}\label{PrB}
F(\overline{\Sigma}_\omega\setminus \{0\})\subset
\overline{\Sigma}_\omega.
\end{equation}
\item [(iv)] For all $\beta\in
(-\pi/2,\pi/2),$
\begin{equation}\label{A11}
{\rm Re}\,F(te^{i\beta})\ge
\cos\beta \, F(t),\qquad t>0.
\end{equation}
Moreover, for every  $c\in (0,1],$
\begin{equation}\label{B1}
F(t)\ge c\, F(c t),\qquad t>0,
\end{equation}
and
\begin{equation}\label{LPB1}
|F(te^{i\beta})|\ge
c\cos\beta \, F(ct),\qquad t>0.
\end{equation}
\end{itemize}
\end{thm}
\begin{proof}
The proofs of (i), (iii) and \eqref{A11} can be found e.g. in \cite{Cauer}, \cite[Corolllary 2]{Goldberg} (or \cite[Theorem 2]{Richards}), and \cite[Theorem 3.4]{Brown}, respectively.  The paper \cite{Brown} contains a unified approach to the proofs of these and similar properties.
The property (iii) goes back to \cite{Brune}, and it is a direct consequence of (i).

The property (ii) is an easy consequence of (i) as well, more general estimates can be found in \cite{Brown}.

To prove (\ref{B1}), it suffices to note that
\[
\frac{t}{t^2+s^2}\ge c \frac{c t}{(c t)^2+s^2},\qquad t,s>0,\quad c\in (0,1],
\]
 hence (\ref{A}) implies (\ref{B1}).
The statement (\ref{LPB1}) follows from
$|F(te^{i\beta})|\ge {\rm Re}\,F(te^{i\beta})$, $t>0$,
and (\ref{A11}), (\ref{B1}).
\end{proof}

Let $\theta_1,\theta_2\in (0,\pi]$. We say that $f\in
\mathcal{NP_+}(\theta_1,\theta_2)$ if $f$ is holomorphic in
$\Sigma_{\theta_1}$ and, moreover
\[
f:\,(0,\infty)\to [0,\infty) \quad \text{and} \quad
f(\Sigma_{\theta_1})\subset \overline{\Sigma}_{\theta_2}.
\]
Denote $\mathcal{NP_+}(\theta):= \mathcal{NP_+}(\theta,\theta)$ so
that $\mathcal{NP_+}= \mathcal{NP_+}(\pi/2)$.

Observe that  that $f\in \mathcal{NP_+}(\theta_1,\theta_2)$ if and
only if
\begin{equation}\label{fmapF}
F(\lambda):= [f(\lambda^{2{\theta_1}/\pi})]^{\pi/(2\theta_2)}\in
\mathcal{NP_+}.
\end{equation}

The next corollary provides a lower bound for $f \in \mathcal{NP_+}(\theta_1,\theta_2)$
in terms of the restriction of $f$ to $(0,\infty).$

\begin{cor}\label{NPabsd}
Let $f \in \mathcal{NP_+}(\theta_1,\theta_2)$ for some
$\theta_1,\theta_2\in (0,\pi]$. Then for every $\theta\in
[0,\theta_1),$
\begin{equation}\label{NP101}
f(\overline{\Sigma}_\theta\setminus \{0\})\subset
\overline{\Sigma}_{\theta \theta_2/\theta_1},
\end{equation}
and for all $c\in (0,1]$ and $\beta\in (-\theta_1,\theta_1),$
\begin{equation}\label{NP102}
|f(te^{i\beta})|\ge
c^{2\theta_2/\pi}\cos^{2\theta_2/\pi}\left(\frac{\pi\beta}{2\theta_1}
\right)f(c^{2\theta_1/\pi} t),\qquad t>0.
\end{equation}
\end{cor}

\begin{proof}
Since, $f(\lambda)=[F(\lambda^{\pi/(2\theta_1)})]^{2\theta_2/\pi}, \lambda \in \Sigma_{\theta_1},$ where $F$ is defined by \eqref{fmapF},
the statement (\ref{NP101}) follows directly from
 \eqref{PrB}. Next, if $t>0$ and $\beta\in (0,\theta_1)$,
then by (\ref{LPB1})  for   every $c\in (0,1],$
\begin{align*}
|f(te^{i\beta})|=&|F(t^{\pi/(2\theta_1)}e^{i\pi\beta/(2\theta_1)})|^{2\theta_2/\pi}\\
\ge&
c^{2\theta_2/\pi}\cos^{2\theta_2/\pi}(\pi\beta/(2\theta_1))[F(c
t^{\pi/(2\theta_1)})]^{2\theta_2/\pi}\\
=&c^{2\theta_2/\pi}\cos^{2\theta_2/\pi}(\pi\beta/(2\theta_1))f(c^{2\theta_1/\pi}
t).
\end{align*}
\end{proof}

The \emph{subclass}  of $\mathcal{\mathcal{NP_+}}$ formed by
complete Berntsein functions and denoted by $\mathcal{CBF}$ will
also be important in our considerations. Complete Bernstein
functions allow a number of equivalent characterizations. The
following one, which can serve as the definition of a complete
Bernstein function, can be found in \cite[Theorem
6.2]{SchilSonVon2010}. We  say that the function $\psi :
(0,\infty)\mapsto [0,\infty)$ is \emph{complete Bernstein} if
$\psi$ admits  an analytic continuation to $\C\setminus
(-\infty,0]$ which is given by
\begin{equation}\label{Cbf}
\psi(\lambda)=a+b\lambda+\int_{(0,\infty)}
\frac{\lambda\,\mu(ds)}{\lambda+s},
\end{equation}
where $a,b\ge 0$ are non-negative constants and $\mu$ is a
positive Borel measure on $(0,\infty)$ such that
\begin{equation}\label{mu}
\int_{(0,\infty)} \frac{\mu(ds)}{1+s}<\infty.
\end{equation}
Given $\psi,$ the triple $(a,b,\mu)$ is determined uniquely, and it is called the
Stieltjes representation of $\psi.$ The standard examples of
complete Bernstein functions include $\lambda^\alpha, \alpha \in
[0,1],\log(1+\lambda)$ and $\lambda/(\lambda+a), a >0.$

The class $\mathcal{CBF}$ has a rich structure which is
particularly suitable for functional calculi purposes. We will
need just a few of them and refer to  \cite[Sections 6 and
7]{SchilSonVon2010} for a comprehensive account. In particular,
the following elementary properties of ${\mathcal CBF}$ will be
useful, see e.g. \cite[Theorem 6.2 and Corollary
7.6]{SchilSonVon2010} for their discussion.
\begin{thm}\label{cbfprop}
\begin{enumerate} \item [(i)] Let $\psi$ be a holomorphic function in $\mathbb C \setminus (-\infty, 0].$ Then $\psi \in \mathcal{CBF}$ if and only if
 $\psi \left(\C^+ \right)\subset \overline{\C^+}, \psi((0,\infty))\subset[0,\infty),$
and
$
 \text{there exists}\,\, \psi(0+)=\lim_{\lambda\to 0+}\,\psi(\lambda);
$
\item [(ii)] Let $\psi, \varphi \in \mathcal{CBF}.$ Then $\psi+\varphi, \psi\circ \varphi \in \mathcal {CBF}.$
\end{enumerate}
\end{thm}
The following result allows one to bound the imaginary part of a complete Bernstein function by means of
its derivative on the positive half-axis.
\begin{lemma}\label{phi}
Let $\psi\in \mathcal{CBF}$. Then for all  $\beta\in(-\pi,\pi)$ and
$t>0$,
\begin{equation}\label{R101x}
{\rm Im}\,\psi(te^{i\beta})\le 2t\tan(\beta/2)\,\psi'(t).
\end{equation}
\end{lemma}

\begin{proof}
Let $\psi$ be of  the form (\ref{Cbf}) and let $\lambda=te^{i\beta}.$ Then observing that
\[
\psi'(\lambda)=b +\int_{(0,\infty)}
\frac{s\,\mu(ds)}{(\lambda+s)^2},
\]
and using the inequality
\[
|te^{i\beta}+s|^2 \ge \cos^2(\beta/2) (t+s)^2,\qquad \beta\in
(-\pi,\pi),
\]
we obtain  that
\begin{align*}
{\rm Im}\,\psi(te^{i\beta})= &bt\sin\beta+\int_{(0,\infty)} {\rm
Im}\,\frac{t e^{i\beta}}{te^{i\beta}+s}\,\mu(ds)\\
=&t\sin\beta \left(b+ \int_{(0,\infty)}\frac{s\,\mu(ds)}{|t
e^{i\beta}+s|^2}\right) \\
\le& t\sin\beta \left(b+
\frac{1}{\cos^2(\beta/2)}\int_{(0,\infty)}\frac{s\,\mu(ds)}{(t+s)^2}\right)\\
\le & \frac{t\sin\beta}{\cos^2(\beta/2)} \left(b+ \int_{(0,\infty)}\frac{s\,\mu(ds)}{(t+s)^2}\right)\\
=&2t\tan(\beta/2)\psi'(t).
\end{align*}
\end{proof}

Now we introduce a technical condition which will be basic for
estimates in subsequent sections. It is a local version of
\eqref{R101x}.
\begin{defn}\label{classd}
Let $a > 0$ and $\theta \in (0, \pi).$ We let $\mathcal{D}_{\theta}(0,a)$
be the space of holomorphic
functions $f$  on $\Sigma_\theta$ such that
\begin{itemize}
\item [a)] $f$ is real on $(0,\infty)$ and strictly increasing on $(0,a);$
\item [b)] for every $R>0$ and every $\beta \in
(-\theta,\theta)$ there exist $b=b(\beta,R)\in (0,\min(1,a/R))$
and $m=m(\beta)$ such that
\begin{equation}\label{AddCon}
|{\rm Im}\,f(te^{i\beta})|\le m t f'(b t)
\end{equation}
for all $t \in (0,R).$
\end{itemize}
\end{defn}

Observe that  by Lemma \ref{phi}, $\mathcal{CBF} \subset
\mathcal{D}_{\theta}(0,a)$ for all $a >0$ and $\theta \in
(0,\pi).$ The next lemma, proved in Appendix A, shows that some
functions from $\mathcal{NP_+}$ belong to
$\mathcal{D}_{\pi/2}(0,1).$
\begin{lemma}\label{CorD}
Suppose that
\[
\sum_{n=0}^\infty c_n=1, \qquad c_n\ge 0, \,\, n\ge 0,
\]
and let
\[
\mathbf{h}(\lambda):=1-\sum_{n=0}^\infty c_n \left(\frac{1-\lambda}{1+\lambda}\right)^n,\qquad \lambda\in \C_+.
\]
Then $\mathbf{h} \in \mathcal{D}_{\pi/2}(0,1)\cap \mathcal{NP_+}.$ Moreover, the corresponding constants $b=b(\beta,R)$ and $m=m(\beta)$ from Definition
\ref{classd} are given by
$$
b= \frac{\cos \beta}{1+R^2} \qquad \text{and} \qquad m=\frac{\pi}{2}.
$$
\end{lemma}

Finally, we will also need the next geometric proposition proved in \cite{GT}.
\begin{prop}\label{pr7.6} \cite[Proposition 3.6]{GT}.
Assume that for $\psi\in \mathcal{CBF}$ there exists $\omega\in (0,\pi/2)$
such that
\begin{equation}\label{ConM0101}
\psi(\C_{+})\subset\overline{\Sigma}_\omega.
\end{equation}
Define $\omega_0 \in (\pi/2,\pi)$ by
\[
|\cos\omega_0|=\frac{\cot \omega}{\cot \omega + 1},
\]
and for $\theta\in (\pi/2,\omega_0)$ define  $\theta_0 \in (0,\pi/2)$ as
\begin{equation}\label{7.14}
\cot\theta_0=\frac{\cot \omega -(\cot\omega + 1)|\cos\theta|}{\sin\theta}.
\end{equation}
Then
\begin{equation}\label{7.15}
\psi(\overline{\Sigma_\theta})\subset \overline{\Sigma}_{\theta_0}.
\end{equation}
\end{prop}

\subsection{$\Wip-$ and Hausdorff functions}\label{wiphaus}

Let $\Wip(\D)$ be the algebra of holomorphic functions $f$ on
the unit disc $\D$
that have absolutely summable Taylor coefficients:
\[ \Wip(\D) := \left \{ f(\lambda)=\sum_{n=0}^\infty c_n \lambda^n, \,\, \lambda \in \D: \,\, \sum_{n=0}^\infty |c_n| < \infty \right \}.
\]
Clearly, if $f\in \Wip(\D)$ then $f$ is continuous on $\cls{\D}.$
Setting
\[ \norm{f}_{\Wip(\D)} : = \sum_{n=0}^\infty
\abs{c_n}\quad \quad \text{if} \quad f(\lambda) = \sum_{n=0}^\infty c_n \lambda^n,
\]
one infers that $(\Wip(\D), \norm{f}_{\Wip}) $ is a unital commutative Banach
algebra with respect to pointwise multiplication.

Consider now functions $h$ given by
\begin{equation}\label{HaF1}
h(\lambda)=c_0+\sum_{n=1}^\infty c_n\lambda^n, \qquad \text{where}\,\, c_0 \ge
0, \quad c_n:=\int_{[0,1)} t^{n-1} \, \nu(dt), \quad n \ge 1,
\end{equation}
and $\nu$ is a bounded positive Borel measure on $[0,1)$ such that
\begin{equation}\label{proper0}
c_0+\int_{[0,1)}\frac{\nu(dt)}{1-t}=1.
\end{equation}
For the purposes of the present paper, the functions $h$
satisfying \eqref{HaF1} and \eqref{proper0} will be called {\it
regular Hausdorff functions} (since the moment sequences $(c_n)$
are often called Hausdorff sequences). We will write $h \sim (c_0,
\nu)$ and say that the measure $\nu$ is the (Hausdorff) {\it
representing measure} for $h$.

Observe that if $h$ is defined by  \eqref{HaF1} and \eqref{proper0} then
\begin{equation}
h(\lambda)=c_0+\int_{[0,1)} t^{-1}\left(\sum_{n=1}^\infty (t\lambda)^n \right)
\nu(dt)=c_0+\int_{[0,1)}\frac{\lambda\nu(dt)}{1-t\lambda},\qquad \lambda \in \D,
\label{fdis1}
\end{equation}
and moreover $h$ extends analytically to  $\lambda\in \C\setminus[1,\infty).$
By \eqref{proper0} and Fatou's lemma,
\begin{equation}\label{disc}
\|h\|_{\Wip(\D)}=h(1)=\sum_{n=0}^\infty c_n=1.
\end{equation}
The next proposition relates the regular Hausdorff functions to
complete Bernstein functions thus connecting the discrete and the
continuous settings.

\begin{prop}\label{prop1}
Let $h\sim (c_0,\nu)$ be a regular Hausdorff function, and let
\begin{equation}\label{dd}
\psi(\lambda):=1-h(1-\lambda), \qquad \lambda \in \D.
\end{equation}
Then $\psi$ extends to a complete Bernstein function of the form $(0,b,\mu),$
where
\[b=\nu(\{0\}), \qquad \mu(dt)=\frac{\nu (ds) - b \delta_0(ds)}{s(1-s)}\qquad (t=(1-s)/s),
\]
and $\delta_0$ denotes the Dirac measure at $0.$

Conversely,
suppose $\psi\in \mathcal{CBF}$ is such that $\psi\sim (0,0,\mu)$.
Then
\[
h(\lambda):=\psi(1)-\psi(1-\lambda)
\]
is a regular Hausdorff function such that $h\sim(0,\nu)$, where
$$
\nu(ds)=\frac{t\mu(dt)}{(1+t)^2} \qquad  (s=1/(1+t)).
$$
\end{prop}

\begin{proof}
Let $\nu(ds)=b\delta_0(ds)+\nu_0(ds)$, where $\nu_0(ds)$ is a Borel
measure on $(0,1)$. Taking into account \eqref{HaF1}, we have
\begin{align*}
\psi(\lambda)=&\int_{[0,1)} \frac{\nu(ds)}{1-s}-\int_{[0,1)}
\frac{(1-\lambda)\,\nu(ds)}{1-s+\lambda s} =\int_{[0,1)}
\frac{\lambda\,\nu(ds)}{(1-s+\lambda s)(1-s)}\\
=&b \lambda+\int_{(0,1)} \frac{\lambda\,\nu_0(ds)}{((1-s)/s+\lambda)s(1-s)}.
\end{align*}
So, passing to the push-forward measure $\mu(dt)$ of $\frac{\nu_0(ds)}{s(1-s)}$ under the map $t: (0,1)\to (0,\infty),
t(s)=(1-s)/s,$
we obtain
that
\[
\psi(\lambda)=b \lambda+\int_{(0,\infty)}  \frac{\lambda\,\mu(dt)}{t+\lambda}, \qquad
\mu(dt)=\frac{(t+1)^2\,\nu_0(ds)}{t},
\]
and
\begin{equation}\label{meas}
\int_{(0,\infty)}\frac{d\mu(t)}{1+t}=\int_{(0,1)}
\frac{\nu(ds)}{1-s} <\infty.
\end{equation}

If $\psi\in \mathcal{CBF}$ and $\psi\sim (0,0,\mu),$ then
\begin{align*}
 \psi(1)-\psi(1-\lambda)=&\int_{(0,\infty)} \frac{\mu(dt)}{1+t}
-\int_{(0,\infty)} \frac{(1-\lambda)\,\mu(dt)}{1-\lambda+t}\\
 =&\int_{(0,\infty)}
\frac{\lambda t\,\mu(dt)}{(1+t-\lambda)(1+t)}\\
 =&\int_{(0,\infty)}
\frac{\lambda t\,\mu(dt)}{(1-\lambda/(1+t))(1+t)^2}.
\end{align*}
Passing as above to the push-forward measure $\nu(ds)$ of $\frac{t\mu (dt)}{(1+t)^2}$ under the map $s:(0,\infty)\to (0,1), s(t)=1/(1+t),$ we
obtain that
\[
h(\lambda)=\psi(1)-\psi(1-\lambda)=\int_{(0,1)} \frac{\lambda\,\nu(ds)}{1-s\lambda},
\]
and \eqref{meas} holds.
\end{proof}

To illustrate the second statement in Proposition \ref{prop1} and
in view of further applications in Section \ref{improvesec}, let
us consider the next simple example.

\begin{example}\label{discex}
a) Let $\alpha \in (0,1)$ be fixed.  Recall  that (see \cite[p.
304]{SchilSonVon2010}) $\psi_\alpha(\lambda):=\lambda^{\alpha} \in
\mathcal{CBF}$ and
\[
\psi_\alpha(\lambda)=\frac{\sin(\pi\alpha)}{\pi} \int_0^\infty \frac{\lambda
ds}{(\lambda+s)s^{1-\alpha}}, \qquad \lambda \in \mathbb C\setminus
(-\infty,0].
\]
Thus, $\psi_\alpha(1)=1$ and  $\psi_\alpha \sim(0,0,\mu_\alpha)$,
where
\[
\mu_\alpha(ds)=\frac{\sin(\pi\alpha)}{\pi}
\frac{ds}{s^{1-\alpha}}.
\]
Then, by Proposition \ref{prop1}, $h_\alpha(\lambda):=1-(1-\lambda)^\alpha, \lambda
\in \D,$ is a regular Hausdorff function and $h_\alpha \sim
(0,\nu_\alpha),$ where
$$
\nu_\alpha (dt)=\frac{\sin(\pi\alpha)}{\pi}
\frac{\lambda(1-t)^\alpha\,dt}{(1-\lambda t)t^\alpha}.
$$

b) For $\psi(\lambda):=\frac{\lambda-1}{\log \lambda} \in \mathcal{CBF},$ we use the
representation (see \cite[p. 322]{SchilSonVon2010})
\[
\psi(\lambda)=\int_0^\infty \frac{\lambda\,(s+1)\,ds}{(\lambda+s)s(\log^2s+\pi^2)},
\qquad \lambda \in \mathbb C\setminus (-\infty,0],
\]
so that $\psi\sim (0,0,\mu)$, $\psi(1)=1$, where
$$ \mu (ds) = \frac{(s+1)ds}{s(\log^2s+\pi^2)}.$$
 If $
h(\lambda)=1-\psi(1-\lambda)=1+\lambda/\log(1-\lambda),$ $\lambda \in \D,$ then, by Proposition
\ref{prop1} we infer that $h$ is a regular Hausdorff function  and
$h\sim (0,\nu),$ where
\[
\nu(dt)=\frac{dt}{t(\log^2(1/t-1)+\pi^2)}.
\]
\end{example}

In Example \ref{discex}, a) and b) one may also write down the
Taylor coefficients for $h$  explicitly.

\section{Sectorial operators and Ritt operators}\label{secritt}

In this section we will introduce and discuss sectorial and Ritt
operators, the main objects of our studies. Moreover, we recall
and study the notion of Stolz domain. This is a geometric notion
related to Ritt operators and to some extent matching the notion
of sector for sectorial operators. Moreover, we prove several
geometric properties of the spectrum of Ritt operators.

Let us first recall that a closed, densely defined  operator $A$
is called sectorial with sectoriality angle $\alpha\in [0,\pi)$ if
$\sigma(A)\subset\overline{\Sigma}_\alpha,$ and for any $\omega\in
(\alpha,\pi)$ exists $M(A,\omega)<\infty$ such that
\[
\| z (z-A)^{-1}\|\le M(A,\omega), \qquad z\not\in
\overline{\Sigma}_\omega.
\]
The set of the sectorial operators with angle $\alpha\in [0,\pi)$
will be denoted by $\mbox{Sect}(\alpha)$. Note that $A\in
\Sect(\alpha)$ for some $\alpha\in [0,\pi)$ if and only
\begin{equation}\label{MQ}
M(A):=\sup_{z>0}\,\|z (z+A)^{-1}\|<\infty.
\end{equation}
Define also the minimal angle of sectoriality $\alpha(A)$ of a sectorial operator $A$ as
$$
\alpha(A):=\inf\{\alpha : A \in \mbox{Sect}(\alpha)\}.
$$
(Note that $\inf$ above can never be replaced by $\min.$)
In this paper, we will mostly be dealing with bounded sectorial operators,
although some operators will a priori be considered as unbounded ones.

As was explained in the introduction, the theory of Ritt operators is
well-developed by now and there are many papers treating various
aspects of such operators. Being unable to present all important
and relevant results, we thus restrict ourselves to discussing
only very basic aspects of that theory.

Let
us first recall that  $T\in \mathcal{L}(X)$ is said to be
\emph{Ritt} if $\sigma(T) \subset \overline \D$  and there exists
$C\ge 1$ such that
\begin{equation}\label{RRR0}
\|(z-T)^{-1}\|\le \frac{C}{|z-1|},\qquad z\in \C
\setminus \overline \D.
\end{equation}
Note that
if $C=1$ in \eqref{RRR0} then necessarily $\sigma(T)=\{1\},$ see
\cite[p. 154]{Lyubich}. There is a direct link between the notion
of Ritt operators and the notion of sectorial operators. Recall
that $T\in \mathcal{L}(X)$ is Ritt if and only if
$\sigma(T)\subset \D \cup\{1\}$ and there is $\omega \in
[0,\pi/2)$ such that the semigroup $(e^{-(1-T)z})_{z \in \C}$ is
sectorially bounded in $\Sigma_\omega,$ see e.g. \cite[Th.
1.5]{Dungey} and the comments preceding it.
This fact allows the following convenient reformulation which we separate for future references.
\begin{thm}\label{sectorritt}
An operator $T \in {\mathcal L}(X)$ is  Ritt if and only if there exists $\alpha \in [0,\pi/2)$
such that
\[
\sigma(T)\subset (\D\cup\{1\}) \cap \{z\in \C: 1-z
\in\overline{\Sigma}_{\alpha}\} \quad \text{and} \quad (1-T)\in
\Sect(\alpha).
\]
\end{thm}

Observe  that the last condition means that for any $\beta\in
(\alpha,\pi/2)$ there exists $C_\beta \ge 1$ such that
\begin{equation}\label{RRR}
\|(z-T)^{-1}\|\le \frac{C_\beta}{|z-1|},\qquad z \in \C\setminus
\left((1- \overline{\Sigma}_{\beta}) \cap \overline{\D} \right).
\end{equation}
Thus, if \eqref{RRR} holds, we will say that \emph{$T$ is a Ritt
operator of angle $\alpha.$}

Note that the Ritt condition \eqref{RRR0} has a number of
implications for the shape of the spectrum of $T.$ To formulate
them we need to define several concepts.

For $\sigma\ge 1$  define a \emph{Stolz domain} $S_\sigma$ by
\begin{equation}\label{San}
S_\sigma:=\{z\in \D:\,|1-z|/(1-|z|) < \sigma\}\cup\{1\},
\end{equation}
Clearly, $S_\sigma=\{1\}$ if $\sigma=1.$

To relate Stolz domains to angular sectors, observe that
\begin{equation}\label{stolsec}
 1-\overline{S}_\sigma \subset \overline{\Sigma}_\omega,\qquad \omega=\arccos(1/\sigma).
\end{equation}
Indeed, let $\sigma>1$ and $1\not=z=1-\rho e^{i\alpha}\in S_\sigma\subset \D$. Then
$\rho<\cos\alpha\le 1$ and
$\rho/(1-|1-\rho e^{i\alpha}|)< \sigma$,
or
\begin{equation}\label{StAa}
\sigma\,|1-\rho e^{i\alpha}|< \sigma-\rho.
\end{equation}
A direct calculation shows that
\begin{equation}\label{StAa1}
\rho < \frac{2\sigma}{\sigma^2-1}(\sigma\cos\alpha-1).
\end{equation}
Therefore, we have, in particular, that
\[
\cos\alpha> \frac{1}{\sigma}\qquad \mbox{and}\qquad
(1-S_\sigma)\setminus\{0\}\subset \Sigma_\omega.
\]
Remark that  $1-\overline{S}_\sigma$ is not a subset of
$\overline{\Sigma}_{\tilde{\omega}}$ for any
$\tilde{\omega}<\omega$.

The next result sharpens the definition of Ritt operators in terms
of Stolz domains.

\begin{prop}\label{rittchar}
Let $T$ be a Ritt operator satisfying \eqref{RRR0} for some $C \ge 1.$  Then
\begin{equation}\label{widmo}
\sigma(T)\subset \overline{S}_\sigma \quad \text{with} \,\, \sigma=C,
\end{equation}
and for any $\delta>\sigma$ there exists $C_\delta$ such that
\begin{equation}\label{dopE}
\|(1-z)(z-T)^{-1}\|\le C_\delta,\qquad z \in \C\setminus
{S}_\delta.
\end{equation}
Conversely, if \eqref{widmo} and \eqref{dopE} hold for some
$\sigma \ge 1,$ then $T$ is Ritt.
\end{prop}

\begin{proof}
Note first that if $C=1$ then $\sigma(T)=\{1\}$ so that
\eqref{dopE} holds, see \cite[p. 154]{Lyubich}. Assume now that
$T$ is a Ritt operator satisfying \eqref{RRR0} with $C>1.$ Then,
by \cite[Proposition 1, Theorem 2 and Corollary]{Lyubich},
\[
\sigma(T)\subset \Omega(q):=\{z \in
\D\cup\{1\}:\,|z-e^{i\varphi}|\ge q|1-e^{i\varphi}|\;\; \text{for
all}\,\, \varphi\in [0,2\pi)\},
\]
where $q=\frac{1}{C}.$ Moreover,  $\Omega(q)$ is a closed convex
set contained in the shifted  sector $
1-\overline{\Sigma}_{\arccos q}, $ and for any $\delta\in (\arccos
q,\pi/2),$
\begin{equation}\label{Est39}
\|(z-T)^{-1}\|\le \frac{C(\delta)}{|z-1|},\quad
1-z\not \in \overline{\Sigma}_\delta,
\end{equation}
where $C(\delta)=\frac{C}{1-C\cos\delta}.$ By Lemma \ref{Lfir}
(proved in Appendix A), the set $\Omega(q)$ can be described as
\[
\Omega(q)\setminus\{1\}=\left\{z\in \D:
\frac{1-|z|^2}{2|1-z|}\ge q\right\}.
\]
Then, since
\[
\frac{1-|z|^2}{2|1-z|}\le \frac{1-|z|}{|1-z|},\qquad z\in\D,
\]
 the definition \eqref{San} of Stolz domain  yields
\[
\Omega(q)\subset \overline{S}_{\sigma},\qquad \sigma=1/q=C,
\]
i.e. \eqref{widmo} holds. Then \eqref{dopE}
follows from \eqref{stolsec},
\eqref{widmo} and \eqref{Est39}.

The converse implication follows from the obvious fact that
$S_\sigma \subset \D \cup \{1 \}$  for all $\sigma \ge 1$ and a
characterization of Ritt operators in terms their sectoriality
given in Theorem \eqref{sectorritt}  (see e.g. \cite[Theorem 1.5]{Dungey}).
\end{proof}

Proposition \ref{rittchar} motivates the following definition.  An
operator  $T\in \mathcal{L}(X)$ is said to be \emph{Ritt operator
of Stolz type} $\sigma \in [1,\infty)$ if $\sigma (T) \subset
\overline{S}_\sigma$  and $T$ satisfies (\ref{dopE}) for any
$\delta
>\sigma.$

\begin{remark}
Note that there is an alternative geometric object related to Ritt
operators. Namely, define a set $B_\omega$, $\omega\in (0,\pi/2)$,
as the interior of the convex hull of $1$ and the disc
$D_{\sin\omega}:=\{z\in \C:\,|z|<\sin\omega\}$, i.e.
\[
\overline{B}_\omega=\overline{{\rm co}}\,(D_{\sin\omega}\cup\{1\}).
\]
In \cite{Merdy}, it is $B_\omega$ that is called a Stolz domain,
while we use that terminology for $S_\sigma.$  Note that $B_\omega
\subset 1-\Sigma_\omega.$ One can prove that $T \in {\mathcal
L}(X)$ is Ritt if and only if there exists $\alpha \in (0,\pi/2)$
such that $\sigma (T)\subset {\overline B}_{\alpha}$ and for any
$\beta \in (\alpha,\pi/2)$ the set $\{(z-1)(z-T)^{-1}: z \in
\C\setminus {\overline B}_{\beta}\}$ is bounded. See e.g.
\cite[Definition 2.2]{Merdy} and \cite[Lemma 2.1]{Merdy}
concerning the above. However, the domains as $B_\omega$  appear
to be less convenient for the study of the permanence properties
of Ritt operators under functional calculi. Thus, we do not
discuss them in this paper.
\end{remark}

The relevance of Stolz domains is also clear from the statement
given below which will be instrumental in the proof of our main assertion.

\begin{prop}\label{Stangle1}
Let $h(\lambda)=\sum_{n=0}^{\infty}c_n \lambda^n, \lambda \in \overline{\D}, c_n \ge
0,$ be such that $\sum_{n=0}^{\infty}c_n=1$. Then for each $\sigma
\ge 1,$
\begin{equation}\label{RelSt}
h(\overline{S}_\sigma)\subset \overline{S}_\sigma.
\end{equation}
\end{prop}

\begin{proof}
Since
\[
\frac{|1-\lambda^n|}{1-|\lambda^n|}= \frac {|\sum_{k=0}^{n-1}\lambda^k|}{\sum_{k=0}^{n-1} |\lambda|^k}\cdot \frac{|1-\lambda|}{1-|\lambda|}\le
\frac{|1-\lambda|}{1-|\lambda|},\quad \lambda\in \D,
\]
each of the functions $h_n(\lambda):=\lambda^n$, $n\in \Z_+,$
satisfies the relation \eqref{RelSt}.  Then, by the convexity
$\overline{S}_\sigma$, the inclusion \eqref{RelSt} holds for any $h$ given by the convex power series
$\sum_{n=0}^{\infty}c_n \lambda^n.$
\end{proof}

\subsection{Operator Cayley transform and its relation to Stolz domains}

In this subsection, we will discuss the operator Cayley transform which
will our basic tool in reducing considerations in the discrete
setting to their half-plane analogues. However, as we already remarked in the
introduction, as far as Ritt operators are bounded, the discrete situation has its specifics so
that it makes a sense to study it in some more details.

As far as we will be aiming at reducing the arguments for the unit disc
to the half-plane setting, the \emph{Cayley transform} $\mathcal{C}$ will clearly play a crucial role.
Recall that the Cayley transform is given by
\begin{equation}\label{funU}
\mathcal{C}(\lambda):=\frac{1-\lambda}{1+\lambda},\quad \lambda\not=-1,
\end{equation}
and that $\mathcal{C}$  maps  $\D$ onto $\C_{+}$ conformally.

The following proposition relates Stolz domains and angular
sectors via the Cayley transform, and will be useful in the
sequel.

\begin{prop}\label{Stangle}
Let $\mathcal{C}$ be the Cayley transform. If  $\sigma\ge 1$ and $\omega=\arccos(1/\sigma),$ then
\[
\mathcal{C}(\overline{S}_\sigma)\subset \overline{\Sigma}_\omega.
\]
\end{prop}

\begin{proof}
Let $\lambda=1-\rho e^{i\alpha}\in S_\sigma, \lambda \not = 1$. Then
\[
\mathcal{C}(\lambda)=\frac{\rho e^{i\alpha}}{2-\rho e^{i\alpha}}=
\frac{\rho (2e^{i\alpha}-\rho)}{|2-\rho e^{i\alpha}|^2}.
\]
Using \eqref{StAa} and \eqref{StAa1}, we obtain
\begin{eqnarray*}
\frac{{\rm Re}\,\mathcal{C}(\lambda)}{|\mathcal{C}(\lambda)|}&=&
\frac{2\cos\alpha-\rho}{|2-\rho e^{i\alpha}|}\\
&\ge& \frac{2\cos\alpha-\rho}{1+|1-\rho e^{i\alpha}|}\\
&\ge& \frac{2\cos\alpha-\rho}{1+(\sigma-\rho)/\sigma}\\
&=&\frac{1}{\sigma}+
\frac{2\sigma(\sigma\cos\alpha-1)-(\sigma^2-1)\rho}{\sigma (2\sigma-\rho)}\\
&\ge& \frac{1}{\sigma},
\end{eqnarray*}
that is $\mathcal{C}(1-\rho e^{i\alpha}) \in
\overline{\Sigma}_\omega$, where $\omega=\arccos (1/\sigma)$.
\end{proof}

Now we turn to the operator analogue of $\mathcal C.$
For   $T \in {\mathcal L}(X)$ with $\sigma(T)\subset \overline{\D},$ and $\ker(1+T)=\{0\}$ we define the {\it Cayley transform}
$\mathcal{C}(T)$ as
\begin{equation}\label{Cayley}
\mathcal{C}(T):=(1-T)(1+T)^{-1}.
\end{equation}
If ${\rm ran}\, (1+T)$ is dense in $X$ then it is straightforward
that $\mathcal{C}(T)$ is a closed densely defined operator on $X$
and  $\sigma(\mathcal{C}(T))\subset \overline{\C}_{+}$. In our
considerations, we will always deal with $T$ such that    $-1 \not
\in \sigma (T).$ Thus, in the sequel, ${\mathcal C}(T)$ will
always be \emph{bounded.}

Finally note that, by a direct calculation,
\begin{equation}\label{ATeq}
\mathcal{C}(\mathcal{C}(T))=T.
\end{equation}

The following simple proposition  relates sectoriality of $T$ to
that of $\mathcal{C}(T).$
\begin{prop}\label{Lm}
Let $T$ be a power-bounded operator on $X$ such that $-1\notin\sigma(T),$ and let $\sup_{n\ge
0}\,\|T^n\|:=M$. Then for any
$\beta\in (\pi/2,\pi),$
\begin{equation}\label{A0}
\|(\mathcal{C}(T)-z)^{-1}\|\le \frac{3M(1+\|T\|)}{|z
\cos\beta|},\quad z\not\in \overline{\Sigma}_\beta.
\end{equation}
In particular, $\mathcal{C}(T)\in \Sect(\pi/2)$.
\end{prop}

\begin{proof}
Note first that if $ \lambda\not\in \sigma(T)$ and $
z=\frac{1-\lambda}{1+\lambda}$ then
\begin{equation}\label{pinp}
({\mathcal C}(T)-z)^{-1} =-\frac{(1+\lambda)}{2}(1+T)(T-\lambda)^{-1}.
\end{equation}
Since for $z\not\in \overline{\C}_{+}$ one has $
z=\frac{1-\lambda}{1+\lambda}$ with $ \lambda=\frac{1-z}{1+z}\not\in
\overline{\D},$ the identity (\ref{pinp}) yields
\begin{equation}\label{4.18}
\|({\mathcal C}(T)-z)^{-1}\| \le
\frac{|1+\lambda|}{2}\|(1+T)(T-\lambda)^{-1}\|.
\end{equation}
By assumption and the Neumann series expansion we have
\begin{equation}\label{4.16}
\|(T-\lambda)^{-1}\|\le \frac{M}{|\lambda|-1},\quad \lambda\in \C,\quad
|\lambda|>1,
\end{equation}
hence if $z\not\in \overline{\C}_{+},$ then
\begin{align}
|1+\lambda|\|(T-\lambda)^{-1}\|\le& M\frac{|1+\lambda|}{|\lambda|-1}\label{4.19}\\
=&\frac{2M}{|1-z|-|1+z|}\notag \\
=&M\frac{|1-z|+|1+z|}{2|{\rm Re}\,z|}\notag \\
 \le&
M\frac{1+|z|}{|z \cos\beta|}.\notag
\end{align}
Thus, from (\ref{4.18}) and (\ref{4.19}) it follows that if $
z\not\in \overline{\C}_{+}$ is such that $|z|<a,$ then
\begin{equation}\label{3.16}
\|({\mathcal C}(T)-z)^{-1}\|\le M\frac{(a+1)}{2|z \cos\beta|}
\end{equation}
 Next, if  $z\not\in \overline{\C}_{+}$ satisfies
$|z|\ge a>1$, then
\[
\frac{1+|z|}{|z|}\le \frac{a+1}{a},\quad\mbox{and}\quad
\frac{|1+\lambda|}{2}=\frac{1}{|1+z|}\le \frac{a}{(a-1)|z|},
\]
so using (\ref{4.18}) and (\ref{4.19})  and observing that
$$(1+T)(T-\lambda)^{-1}=1+(\lambda+1)(T-\lambda)^{-1},$$ we obtain
\begin{align}
\|({\mathcal C}(T)-z)^{-1}\| \le&
\frac{|1+\lambda|}{2}\left(1+|\lambda+1|\|(T-\lambda)^{-1}\|\right) \label{4.20}\\
\le&
\frac{a}{(a-1)|z|}\left(1 +\frac{M(a+1)}{a|\cos\beta|}\right)
 \frac{M(2a+1)}{(a-1)|\cos\beta|}\cdot\frac{1}{|z|}.\notag
\end{align}
Setting finally $a=4$, (\ref{4.19}) and (\ref{4.20}) imply
(\ref{A0}).
\end{proof}

We proceed with revealing an interplay between geometry of the spectrum
of Ritt operators and their  Cayley transforms.
\begin{prop}\label{Stolz}
If $T$ is a Ritt operator of Stolz type $\sigma$, then
$\mathcal{C}(T)\in \Sect(\omega)$ for $\omega=\arccos(1/\sigma)$.
\end{prop}

\begin{proof}
Fix $\tilde\sigma>\sigma.$  By assumption,
\begin{equation}\label{p4.7}
\|(T-\lambda)^{-1}\|\le \frac{C_{\tilde\sigma}}{|\lambda-1|},\quad
\lambda\not\in S_{\tilde\sigma}.
\end{equation}
If $\tilde{\omega}=\arccos(1/\tilde\sigma)$ and $z\not\in
\Sigma_{\tilde\omega}, z \not = 0,$ then  by Proposition
\ref{Stangle} there exists $\lambda\not\in S_{\tilde\sigma}$ such that
$z=(1-\lambda)/(1+\lambda)$. Hence,  by (\ref{pinp}) and (\ref{p4.7}),
\[
\|({\mathcal C}(T)-z)^{-1}\|\le
\frac{|1+\lambda|}{2}\|1+T\|\frac{C_{\tilde\sigma}}{|\lambda-1|}
=\frac{C_{\tilde\sigma}\|1+T\|}{2|z|},
\]
so that $\mathcal{C}(T)\in \Sect(\tilde\omega)$. Since the choice
of $\tilde\sigma>\sigma$ is arbitrary, we conclude that
$\mathcal{C}(T)\in \Sect(\omega)$.
\end{proof}

\section{Functional calculi}

\subsection{Holomorphic calculus and operator complete Bernstein functions}\label{hfcsub}
In this subsection we will set up a holomorphic functional
calculus of sectorial operators and will state several of its
properties important for the sequel. The comprehensive accounts on
the extended holomorphic functional calculus can be found in many
texts including e.g. \cite[Chapter 2]{Haa2006} and \cite[Section 9]{Weis}, but we
still feel that the functional calculi theory is
not a part of general background, so we recall its basic features
important for our exposition in subsequent subsections.

For $\varphi \in (0,\pi)$, let
$\mathcal{O}(\Sigma_\varphi)$ stands for
the space of all holomorphic functions
on $\Sigma_\varphi$. Define
\begin{eqnarray*}
H_0^\infty(\Sigma_\varphi)&:=&
\left\{f\in \mathcal{O}(\Sigma_\varphi) : \text{$|f(\lambda)|\leq C\min(|\lambda|^s, |\lambda|^{-s})$ for some $C,s>0$}\right\},
\end{eqnarray*}
and
\begin{eqnarray*}
\mathcal{B}(\Sigma_\varphi) :=
 \left\{f\in \mathcal{O}(\Sigma_\varphi) : \text{$|f(\lambda)|\leq C\max \left(|\lambda|^s, |\lambda|^{-s} \right)$ for some $C,s>0$}\right\}.
\end{eqnarray*}
Note that $H_0^\infty(\Sigma_\varphi)$ and $\mathcal{B}(\Sigma_\varphi)$ are algebras.

Let $0\le \alpha<\varphi<\pi,$ and let  $A\in
\operatorname{Sect}(\alpha).$ For $f \in
H_0^\infty(\Sigma_\varphi)$ and $\alpha_0 \in (\varphi,\pi),$
define
\begin{equation*}
\Phi(f) = f(A):=\frac{1}{2\pi i}\int_{\partial \Sigma_{\alpha_0}} f(\lambda) (\lambda-A)^{-1}\,d\lambda,
\end{equation*}
where $\Sigma_{\alpha_0}$ is the downward oriented boundary of
$\Sigma_{\alpha_0}$.
 This definition is independent of $\alpha_0$, and
\[
\Phi: H_0^\infty(\Sigma_\varphi)\mapsto \mathcal{L}(X),\qquad \Phi(f)=f(A),
\]
is an algebra homomorphism. Let
$
\tau(\lambda):=\frac{\lambda}{(1+\lambda)^2}.
$
Assume that $A$ is injective so that  $\Phi(\tau)=\tau(A)=A(1+A)^{-2}$ is injective as well.

Since for any  $f \in \mathcal B(\Sigma_\varphi)$ there is $n\in\N$  such that
\begin{equation}\label{holom}
\tau^n f\in H_0^\infty (\Sigma_\varphi),
\end{equation}
we can define a closed operator $f(A)$ as
\begin{equation}\label{hfc}
f(A)=[\tau^n(A)]^{-1}(\tau^n f)(A)= [A(1+A)^{-2}]^{-n}\,(f\cdot \tau^n)(A),
\end{equation}
where
\[
(f\cdot\tau^n)(A):=\frac{1}{2\pi i} \int_{\partial
\Sigma_{\alpha_0}} \frac{\lambda^n
\,f(\lambda)}{(\lambda+1)^{2n}}(\lambda-A)^{-1}\,d\lambda,
\]
according to the above. This definition does not depend on the choice of $n$ as far as \eqref{holom} holds.
A mapping
\[
\Phi_e: \mathcal{B}(\Sigma_\varphi) \mapsto \mathcal{L}(X),\quad \Phi_e(f)=f(A),
\]
is an algebra homomorphism, and it is called \emph{the extended holomorphic functional calculus} for $A$.

  Note that $\Phi_e$ formally depends on a choice of $\varphi$, but the calculi are consistent with an appropriate identification.
Thus we may consider the calculus to be defined on
\[
\mathcal{B}[\Sigma_\alpha]:=\bigcup_{\alpha<\varphi<\pi}\,\mathcal{B}(\Sigma_\varphi).
\]

It is important to note that in view of Theorem \ref{NP},(ii) if $\alpha \in [0,\theta_1)$ and $ f \in \mathcal{NP_+}(\theta_1,\theta_2)$ then any $f$ can be regularized by $\tau^2,$ and so $f(A)$ is defined in the extended
holomorphic functional calculus.

The extended holomorphic calculus is governed by usual calculi rules, see \cite[Section 2.3, 2.4]{Haa2006} for more on that.
The following properties of the calculus will be of particular importance for us.
\begin{prop}\label{fcrules}
\begin{itemize}
\item [(i)] If $f$ and $g$ belong to $\mathcal{B}[\Sigma_\alpha],$ then
the following sum rule and product rule hold:
$$f(A)+g(A)\subset (f+g)(A), \qquad
 f(A)g(A) \subset (fg)(A).
 $$
If $g(A)$ is bounded, then the inclusions
above turn into equalities.
\item [(ii)] Let $f \in \mathcal{B}[\Sigma_{\alpha'}]$ and $g \in \mathcal{B}[\Sigma_\alpha].$ Suppose in addition that $g(\Sigma_\alpha)\subset \Sigma_{\alpha'},$
$g(A)\in \Sect(\alpha'),$ and $g(A)$ is injective. Then
$f\circ g\in \mathcal{B}[\Sigma_\alpha],$ and the composition rule hold:
\begin{equation}\label{comprule}
(f\circ g)(A)=f(g(A)).
\end{equation}
\end{itemize}
\end{prop}

The importance of sectoriality angles is well-illustrated  by the
following classical statement on fractional powers of sectorial
operators relevant for our subsequent arguments.

\begin{prop}\label{PrBBL}
Let $\alpha\in [0,\pi)$ and $q>0$
be such that $q\alpha<\pi$. Then $A^q\in {\rm Sect}(q\alpha)$.
Moreover, there exists $M_q(A)>0$ such that for every $\epsilon
>0$
\begin{equation}\label{epapprox}
\|\lambda(\lambda+(A+\epsilon)^q)^{-1}\|\le M_q(A), \qquad \lambda \in (0,\infty).
\end{equation}
\end{prop}
For a proof of the first part of the proposition  see e.g.
\cite[Proposition 3.1.2]{Haa2006} or \cite[Corollary 3.10]{Berg}.
The estimate \eqref{epapprox} is a direct consequence of \cite[Corollary 3.1.3]{Haa2006}
and \cite[Proposition 2.1.2, f)]{Haa2006}.

Complete Bernstein functions introduced in Subsection \ref{funsec}
fall into the scope of the extended holomorphic functional
calculus. Moreover, such functions can be defined for any
sectorial operator regardless of its angle of sectoriality. Indeed,
every complete Bernstein function extends holomorphically to $\C
\setminus (-\infty,0],$ and  \eqref{Cbf} implies that it has a
sublinear growth in any sector $\Sigma_\alpha,$ $\alpha \in
[0,\pi).$ Identifying a complete Bernstein function with its
holomorphic extension to $\C \setminus (-\infty,0],$ we infer that
it belongs to the extended holomorphic functional calculus for any
sectorial operator $A.$ The definition \eqref{hfc} applies in this
case with $n=2.$ Moreover, the following operator analogue of
\eqref{Cbf} holds, see e.g. \cite[Theorem 3.12 and Section
3]{BaGoTa} for its discussion and proof (as well as for more
details on the holomorphic functional calculus of complete
Bernstein functions). Another approach to operator complete Bernstein functions can be found in
 \cite{Berg} and \cite{Schilling}.

\begin{thm}
Let a complete Bernstein function $\psi$ be given by its Stieltjes representation $(a,b,\mu)$
(see \eqref{Cbf}).
Then for every $x$ from the domain $\dom(A)$ of $A,$
\begin{equation}\label{OCbf}
\psi(A)x =a + bAx + \int_{(0,\infty)} A (A+s)^{-1} x \,\mu(ds).
\end{equation}
Moreover, $\dom(A)$ is a core for $\psi(A).$
\end{thm}
Note that a complete Bernstein function $\psi$ of $A$ can also defined in the framework of
other calculi, e.g. Hille-Phillips functional calculus,
where the assumption $\ker (A)=\{0\}$ is not, in fact, necessary. However, we will not
need this fact in the sequel.

\subsection{$\Wip$-calculus}

Now we turn to a discussion of another calculus tailored to deal
with power bounded operators rather than sectorial ones. If $T$ is
a power bounded operator on $X,$ then we can define  a
$\Wip(\D)$-functional calculus for $T$ which does not require
holomorphicity of functions on $\sigma(T)$ as in the case of  the
holomorphic functional calculus from the previous subsection. The
notion of the Hausdorff function will be crucial in this context.
We will show that the notion is just another face of the notion of
the complete Bernstein function explained in the previous
subsection.

Since $\Wip(\D)$ is a convolution Banach algebra, there is a very
natural way to define a function from $\Wip(\D)$ of $T.$
Namely, if $f(\lambda)=\sum_{n=0}^\infty c_n \lambda^n \in \Wip (\D)$  then we set
$$
f(T)=\sum_{n=0}^\infty c_n T^n.
$$
The mapping
$$
\Phi  : \Wip(\D) \mapsto \Lin(X), \qquad \Phi (f)=f(T),
$$
is a
continuous homomorphism of Banach algebras satisfying
$$
\|\Phi(f)\|\le\left( \sup_{n \ge 0}\|T^n\|\right)\|f\|_{\Wip(\D)}.
$$
It is called \emph{the $\Wip(\D)$-calculus for $T.$}

In what follows, we will need a spectral mapping theorem for
$\Wip(\D)$-calculus. This result can be found e.g. in
\cite[Theorem 2.1]{Dungey}.

\begin{prop}\label{connect}
Let $f \in \Wip (\D)$  and let $T$ be a
power-bounded operator on $X$. Then
\begin{equation}\label{h}
\quad \sigma(f(T))=f(\sigma(T)).
\end{equation}
\end{prop}
\begin{remark}
Recall that
if $T_1$ and $T_2$ are commuting bounded operators on $X$ then
\begin{equation}\label{kato}
{\rm dist}(\sigma(T_1),\sigma(T_2))\le \|T_1-T_2\|,
\end{equation}
where ${\rm dist}(\sigma(T_1),\sigma(T_2))$ stands for the Hausdorff distance between $\sigma(T_1)$ and $\sigma(T_2).$
(See e.g. \cite[Theorem IV.3.6]{Kato}.) The proof of Proposition \ref{connect} given in \cite{Dungey} is based on this result, and the result
 will also be useful in the sequel.
\end{remark}

It is important to observe that since $\Wip(\D)$ includes regular
Hausdorff functions, a  Hausdorff function $h$ of a power bounded
operator $T$ is well-defined in the $\Wip(\D)$-calculus. Moreover,
for $h \sim (c_0, \nu)$ one can prove the operator counterpart of
\eqref{fdis1}:
$$
h(T)=c_0+\int_{[0,1)} T(1-tT)^{-1} \, \nu(dt).
$$
As we will not use this formula in the following, its proof is omitted.

Since we will use the two functional calculi, namely the extended
holomorphic functional calculus and $\Wip(\D)$-calculus, a natural
question is whether these calculi are consistent. To clarify this
issue, note that if $T$ is power bounded, then $1-T$ is sectorial
and by \cite[Proposition 3.2]{HaTo10} the $\Wip(\D)$ -calculus
agrees with the holomorphic functional calculus for sectorial
operators in a sense that for appropriate holomorphic $f$ one has
$f(A)=g(T)$ where $A:=1-T$ has dense range and $g(\lambda)=f(1-\lambda).$

Moreover, if $h\sim (c_0,\nu)$ is a regular Hausdorff function,
then $\psi(\lambda):=1-h(1-\lambda)\in \mathcal{CBF}$ by Proposition
\ref{prop1}. Thus, $\psi(A)$ is defined by the extended
holomorphic functional calculus. On the other hand, $h(T)$ can be
defined by the $\Wip(\D)$-calculus. In view of  the next result
proved in \cite[Proposition 3.2]{HaTo10} and formulated for a
future reference, the two calculi are consistent and lead to the
same operator.
\begin{lemma}\label{HaTol}
Let $h$ be a regular Hausdorff function and let $\psi(\lambda)=1-h(1-\lambda)$
be the corresponding complete Bernstein function given by
Proposition \ref{prop1}. If $T$ is a power bounded operator on $X$
such that $\cls{\ran}\,(1-T)=X,$ then
\[1-h(T)=\psi(A), \qquad A:=1-T, \]
where $h(T)$ is defined by means of the $\Wip (\mathbb
D)$-calculus and $\psi(A)$ is defined by the extended holomorphic
functional calculus.
\end{lemma}
\begin{remark}
Recall that, by the mean ergodic theorem, if $T$ is a power bounded
operator on $X$ and  $\cls{\ran}\,(1-T)=X,$ then $\ker
(1-T)=\{0\}.$
\end{remark}
\begin{remark}
Using the approach of Subsection \ref{hfcsub}, one may also define
the extended $\Wip(\D)$-calculus.  In this way, the extended
$\Wip(\D)$-calculus for $T$ comprises more general Hausdorff
functions of $T.$ However our arguments will not require this
generalization.
\end{remark}

%%%%%%%%%%%%%%%%%%%%%%%%%%%%%%%%%%%%%%%%%%%%%%%
\section{$\Wip$-functions of Ritt operators}
%%%%%%%%%%%%%%%%%%%%%%%%%%%%%%%%%%%%%%%%%%%%%%%

We now turn to deriving estimates for  $(z+f)^{-1}(A)$ where
$f=h\circ \mathcal{C} \in \mathcal{NP}_+$ is a convex power series
$h$ of the Cayley transform $\mathcal{C}$ and $A$ is a sectorial
and \emph{bounded} operator on $X$. We start with obtaining an
integral representation for the ``resolvent'' function
$(z+f)^{-1}.$ This representation will lead to  a similar
representation for $(z+f)^{-1}(A)$, and eventually to the
("uniform") sectoriality of $f(A).$ Finally, if $T$ is Ritt and
$A={\mathcal C}(T)$, then the sectoriality of $f(A)$ with an
appropriate angle will imply that $h(T)$ is Ritt.

It is also important to note that our arguments depend
essentially on a convergence of certain approximations of Ritt and
power bounded operators. Thus all constants in the resolvent
bounds given below have been written explicitly so to reveal their
uniformity with respect to approximation and to keep control over
the convergence issues.
\begin{lemma}\label{In1}
Let $f\in \mathcal{NP}_+(\theta_1,\theta_2).$
If
\[
\alpha\in (0,\theta_1), \quad  q\in \left(\pi/\theta_1,\pi/\alpha\right) \quad
\text{and}\quad  \gamma \in \left(0,
\pi\left(1-\frac{\theta_2}{q\theta_1}\right)\right),
\]
then for every $R>0,$
\begin{align}
(z+f(\lambda))^{-1}
=&\frac{q}{\pi}\int_0^{R^{1/q}}
\frac{{\rm Im}\,f(te^{i\pi/q})\,t^{q-1}\,dt}
{(z+f(te^{i\pi/q}))(z+f(te^{-i\pi/q}))(t^q+\lambda^q)}\label{R11}\\
+&\frac{1}{2\pi i}\int_{|\xi|=R}\,
\frac{d\xi}{(z+f(\xi^{1/q}))(\xi-\lambda^q)},\notag
\end{align}
for all $\lambda\in \Sigma_\alpha$, $|\lambda| < R^{1/q},$ and
$z\in \Sigma_\gamma$.
\end{lemma}

\begin{proof}
Let $\alpha\in (0,\theta_1),$  $q\in (\pi/\theta_1,\pi/\alpha)$  and $R>0$
be fixed.
By Corollary \ref{NPabsd}, for every $\beta\in (q\alpha,\pi)$ and all nonzero
  $\lambda\in \Sigma_\alpha$ and $\xi\in
\partial \Sigma_\beta,$
\begin{equation}\label{chi1}
f(\lambda)\in \overline{\Sigma}_{\alpha\theta_2/\theta_1} \quad
\text{and} \quad
f(\xi^{1/q})\in
\overline{\Sigma}_{\beta\theta_2/(q\theta_1)}.
\end{equation}
Note also that if $\lambda\in \Sigma_\alpha,$ and $z\in
\Sigma_\gamma,$ where $\gamma\in (0,\pi(1-\theta_2/(q\theta_1)))$,
then
\[
\alpha\theta_2/\theta_1+\gamma< \pi \quad \text{and} \quad \gamma+\beta\theta_2/(q\theta_1)<\pi.
\]
Now, by Cauchy's theorem, for every  $R>|\lambda|^q,$
\begin{equation}\label{Ca}
(z+f(\lambda))^{-1}=
\frac{1}{2\pi i}\int_{\partial \Sigma_\beta(R)}\,
\frac{d\xi}{(z+f(\xi^{1/q}))(\xi-\lambda^q)},
\end{equation}
where  $\Sigma_\beta(R):=\Sigma_\beta \cap \{z\in \mathbb C: |z|=R \}.$
Deforming the contour $\Sigma_\beta(R)$ in (\ref{Ca})  to the negative semi-axis, we obtain
\begin{align*}
(z+f(\lambda))^{-1}=&\frac{1}{2\pi i}
\int_{|\xi|=R}\,\frac{d\xi}{(z+f(\xi^{1/q}))(\xi-\lambda^q)}\label{RR}
\\
-&\frac{1}{2\pi i}\int_0^R\,
\frac{ds}{(z+f(s^{1/q}e^{i\pi/q}))(s+\lambda^q)}\notag \\
+&\frac{1}{2\pi i}
\int_0^R\,\frac{ds}{(z+f(s^{1/q}e^{-i\pi/q}))(s+\lambda^q)}\notag\\
=&\frac{1}{2\pi i}\int_{|\xi|=R}\,
\frac{d\xi}{(z+f(\xi^{1/q}))(\xi-\lambda^q)}\notag\\
+&\frac{1}{\pi}\int_0^R\,
\frac{{\rm Im}\,f(s^{1/q}e^{i\pi/q})\,ds}
{(z+f(s^{1/q}e^{i\pi/q}))(z+f(s^{1/q}e^{-i\pi/q}))\notag
(s+\lambda^q)},
\end{align*}
and \eqref{R11} follows.
\end{proof}

The above lemma is in fact the heart of our strategy. Using
the representation \eqref{Ca} containing $f(\xi^{1/q})$ rather
than $f(\xi)$ we are able to deform the integration contour to the
negative half-axis so that to pass to the formula \eqref{R11}
containing ${\rm Im} \, f(te^{i\pi/q}).$ In turn, this latter term
${\rm Im} \, f(te^{i\pi/q}),$ for certain $f \in \mathcal{NP_+},$
allows useful estimates, e.g the one given by Lemma \ref{chi} from
Appendix A. Lemma \ref{chi} provides a way to cancel
singularity of the integrand in \eqref{R11} at $t=0$ and thus helps to
show that the  integral  \eqref{R11} converges absolutely. This is
the key point in obtaining resolvent bounds in Theorem \ref{mainT}
below.

Next, using the preceding result and Theorem \ref{NP}, we prove the sectoriality
of $f(A)$ if
$f \in \mathcal{NP_+}(\theta_1,\theta_2)\cap {\mathcal D}_{\theta_1}(0,a)$ for some $a >0.$ The proof
is based on the integral
representation for the resolvent of $f(A)$. The representation yields
the sectoriality of $f(A)$ by means of Theorem \ref{NP} and bounds on ${\rm Im}\, f$
contained in the definition of ${\mathcal D}_{\theta_1}(0,a).$
 Moreover, we  give a very
explicit bound for the resolvent of $f(A).$ It is important for
subsequent approximation arguments where a certain uniformity of resolvent estimates is required.

On
this way the following lemma from \cite[Appendix B, Proposition
B5]{ABHN01} will also be crucial.
\begin{lemma}\label{resolvent}
Let $A$ be a closed densely defined operator on $X,$ and $U$ be a
connected open subset of $\mathbb C$. Suppose that $U \cap\rho(A)$
is nonempty and that there is a holomorphic function $F : U \to
{\mathcal L}(X)$ such that $\{z \in U \cap\rho(A): F(z) =
(z-A)^{-1}\}$ has a limit point in $U.$ Then $U \subset \rho(A)$
and $F(z) = (z-A)^{-1}$ for all $z \in U.$ \end{lemma}

\begin{thm}\label{mainT}
Let $f\in \mathcal{NP_+}(\theta_1,\theta_2)\cap {\mathcal
D}_{\theta_1}(0,a)$ for some $a >0.$ Let $A \in {\mathcal L}(X)$
be such that $A\in \Sect(\alpha)$, $\alpha\in [0,\theta_1),$ and
$\ker{A}=\{0\}$. Then
\[
f(A)\in \Sect(\tilde{\alpha}),\quad \tilde{\alpha}=\frac{\theta_2}{\theta_1}\cdot\alpha.
\]
Moreover,  for all
$q\in \left(\pi/\theta_1,\pi/\alpha\right)$ and
$\gamma \in \left(0, \pi\left(1-\frac{\theta_2}{q\theta_1}\right)\right)$,
one has
\begin{equation}\label{family}
\|(z + f(A))^{-1}\|\le \frac{c_{q,\gamma}}{|z|},\qquad z\in \Sigma_\gamma,
\end{equation}
where
\begin{equation}\label{constant}
c_{q,\gamma}=
\frac{q M_q(A)m(\pi/q)}{Cb\pi \cos^2((\pi/q+\gamma)/2)}
+\frac{2}{\cos((\pi/q+\gamma)/2)},
\end{equation}
with $M_q(A)$ given by  \eqref{epapprox}, $b=b(\pi/q, 2\|A\|)$ and
$m=m(\pi/q)$ corresponding to $f$ by the definition of
$\mathcal{D}_{\theta_1}(0,a),$ and finally $
C=b^{\theta_2/\theta_1}[\cos (\pi^2/(2\theta_1
q))]^{2\theta_2/\pi}. $
\end{thm}

\begin{proof}
Let  $q\in \left(\pi/\theta_1,\pi/\alpha\right) and $ $\gamma \in
\left(0, \pi\left(1-\frac{\theta_2}{q\theta_1}\right)\right)$  be
fixed, and set $R=2^q\|A^q\|$. Having in mind \eqref{R11}, let us
set formally
\begin{eqnarray}\label{Ra}
R_q(z;f, A)&:=&\frac{q}{\pi}\int_0^{R^{1/q}}
\frac{{\rm Im}\, f(te^{i\pi/q})\,t^{q-1}(A^q+t^q)^{-1}\,dt}
{(z+f(te^{i\pi/q}))(z+f(te^{-i\pi/q}))}\\
&-&\frac{1}{2\pi i}\int_{|\xi|=R}\, \frac{(\xi -A^q)^{-1}\,d\xi}{z+f(\xi^{1/q})},\qquad z\in
\Sigma_\gamma \notag.
\end{eqnarray}
We first prove that $R_q(\cdot;f, A):\Sigma_\gamma
\mapsto  \mathcal{L}(X)$  is a well-defined holomorphic function and
derive a bound for $\|z R_q(z;f, A)\|$ when $z \in \Sigma_\gamma.$

We consider  each of the two terms in \eqref{Ra}
separately. To estimate the first term, note that by
Corollary \ref{NPabsd}
\begin{equation}\label{an1}
f(t e^{i\beta})\in \overline{\Sigma}_{\pi/q},\quad t>0,\quad |\beta|\le \pi/q.
\end{equation}
Then, by Lemma \ref{sum} (from Appendix A) and
Corollary \ref{NPabsd},
for all $c\in (0,1]$ and all $z\in \Sigma_\gamma,$
\begin{align}\label{q30}
|(z+f(t e^{i\pi/q}))(z+f(te^{-i\pi/q}))|
\ge& \cos^2\left((\pi/q+\gamma)/2\right)\, \left(|z|+|f(te^{i\pi/q})|\right)^2 \\
\ge& \cos^2\left((\pi/q+\gamma)/2 \right)\, \left(|z|+C\,f(\delta
t)\right)^2,\notag
\end{align}
where
\[
C=c^{2\theta_2/\pi}[\cos(\pi^2/(2q\theta_1))]^{2\theta_2/\pi} \quad \text{and} \quad
\delta=\delta(c):=c^{2\theta_1/\pi}.
\]
Now,  let $b=b(\pi/q,R^{1/q})$ and $m=m(\pi/q)$ be given for $f$
by the definition of $\mathcal{D}_{\theta_1}(0,a).$ Put
\begin{equation}
c:=b^{\pi/(2\theta_1)}\in (0,1],
\end{equation}
so  that
\[
b=\delta=c^{2\theta_1/\pi} \quad \mbox{and}\quad
C=b^{\theta_2/\theta_1}[\cos(\pi^2/(2q\theta_1))]^{2\theta_2/\pi}.
\]
According to (\ref{AddCon}), we have
\[
|{\rm Im}\,f(t e^{i\pi/q})|\le m(\pi/q) t f'(bt),\quad t\in (0,R^{1/q}).
\]
Furthermore, by Proposition \ref{PrBBL}, $A^q$ is sectorial, and
by \eqref{epapprox},
\[
\|(A^q+t^q)^{-1}\|\le \frac{M_q(A)}{t^q},\qquad t>0.
\]
Taking the above bounds into account, we then proceed as follows:
\begin{eqnarray}\label{BcA}
&&\int_0^{R^{1/q}}\,
\frac{|{\rm Im}\,f(t e^{i\pi/q})|\,\|(A^q+t^q)^{-1}\|t^{q-1}\,dt}
{|(z+f(te^{i\pi/q}))(z+f(te^{-i\pi/q}))|} \\
&\le& M_q\int_0^{R^{1/q}}\,
\frac{|{\rm Im}\,f(t e^{i\pi/q})|\,dt}
{|(z+f(te^{i\pi/q}))(z+f(te^{-i\pi/q}))|t}\notag\\
&\le&
\frac{M_q(A) m(\pi/q)}{\cos^2((\pi/q+\gamma)/2)}
\int_0^{R^{1/q}}\,
\frac{f'(b t)\,dt}
{(|z|+C\,f(b t))^2}\notag \\
&\le&\frac{M_q(A) m(\pi/q)}{C b \cos^2((\pi/q+\gamma)/2)}\cdot
\frac{1}{|z|}.\notag
\end{eqnarray}

Next, we  estimate the second term in (\ref{Ra}). Note that
\[
\|(\xi-A^q)^{-1}\|\le \frac{1}{|\xi|-\|A^q\|}\le \frac{2}{|\xi|},\qquad |\xi|\ge 2\|A^q\|.
\]
Using   once again Lemma \ref{sum} and taking into account (\ref{an1}), we infer that
\begin{align}\label{second}
\frac{1}{2\pi}\int_{|\xi|=R}\,
\frac{\|(\xi-A^q)^{-1}\|\,|d\xi|}{|z+f(\xi^{1/q})|}
\le \frac{1}{\pi R\cos\left((\pi/q+\gamma)/2\right)}\int_{|\xi|=R}\,
\frac{|d\xi|}{|z|+|f(\xi^{1/q})|}\\
\le\frac{2}{\cos\left((\pi/q+\gamma)/2\right)}\cdot \frac{1}{|z|},\qquad z\in \Sigma_\gamma.\notag
\end{align}

Finally, since the integrals in \eqref{Ra} converge absolutely,
the operator-valued function  $R_q(\cdot;f, A):\Sigma_\gamma
\mapsto  \mathcal{L}(X)$ is  holomorphic by a standard application of the Morera
theorem.

Thus, due to (\ref{Ra}), (\ref{BcA}) and (\ref{second}),
$R_q(\cdot;f, A): \Sigma_\gamma\mapsto \mathcal{L}(X)$ is
holomorphic for every $\gamma\in (0,\pi(1-q^{-1})).$ Moreover,
\begin{equation}\label{familyT}
\|R_q(z;f,A)\|\le \frac{c_{q,\gamma}}{|z|},\qquad z\in \Sigma_\gamma,
\end{equation}
where $c_{q,\gamma}$ is defined by (\ref{constant}).
(Note that $c_{q,\gamma}$ depends only on $q, \gamma$, $\omega$, $\|A\|$ and $M_q(A).$)

Next we show that if $z\in \Sigma_\gamma,$ then $R_q(\cdot;f, A)$ coincides with $(z+f(A))^{-1},$
and as a consequence that \eqref{family} holds.
From Lemma \ref{resolvent} it follows that it suffices to prove
(\ref{Ra}) only for $z>0$.

So, let $z >0$ be fixed.
Since $A$ has trivial kernel,
and the function $(z+\cdot)^{-1}$  is bounded on $\C_{+}$ for every $z>0$,
 $(z+f)^{-1}(A)$ is  defined
in the extended holomorphic calculus via (\ref{hfc}) with $n=1.$
Moreover, using Lemma \ref{In1}, for all  ${\tilde\omega}\in (\alpha,\omega)$
and $z >0,$
\begin{align*}
&\left(\frac{\lambda}{(\lambda+1)^2(z+f)}\right)(A)
=\frac{1}{2\pi i}
\int_{\partial \Sigma_{\tilde{\omega}}}
\frac{\lambda}{(\lambda+1)^2}\frac{(\lambda-A)^{-1}}{(z+f(\lambda))}\,d\lambda\\
=&\frac{q}{2\pi^2 i}\int_0^{R^{1/q}}
\frac{{\rm Im}\,f(te^{i\pi/q})\,t^{q-1}}
{(z+f(te^{i\pi/q}))(z+f(te^{-i\pi/q}))}
\int_{\partial \Sigma_{\tilde{\omega}}}\frac{\lambda (\lambda-A)^{-1}}
{(\lambda+1)^2(\lambda^q+t^q)}\,d\lambda dt\\
+&\frac{1}{(2\pi i)^2}\int_{|\xi|=R}\,\frac{1}{(z+f(\xi^{1/q}))}
\int_{\partial \Sigma_{\tilde\omega}}
\frac{\lambda}{(\lambda+1)^2}\frac{(\lambda-A)^{-1}}{(\xi-\lambda^q)}\,d\lambda\, d\xi\\
=&\frac{q}{\pi}\int_0^{R^{1/q}}
\frac{{\rm Im}\,f(te^{i\pi/q})\,t^{q-1}}
{(z+f(te^{i\pi/q}))(z+f(te^{-i\pi/q}))}
A(A+1)^{-2}(A^q+t^q)^{-1}\, dt\\
-&\frac{1}{2\pi i}\int_{|\xi|=R}\,\frac{1}{(z+f(\xi^{1/q}))}
A(A+1)^{-2}(\xi-A^q)^{-1}\, d\xi\\
=&A(A+1)^{-2}R_q(z;f,A).
\end{align*}
Hence, by (\ref{hfc}),
\[
(z+f)^{-1}(A)=R_q(z;f,A)
\]
for all $z >0.$ Now the product rule for the extended holomorphic functional
calculus (Proposition \ref{fcrules}, (i)) yields
\begin{eqnarray*}
R_q(z;f,A)(z+f(A))&\subset& (z+f(A))R_q(z;f,A)\\
&=&(z+f)(A)(z+f)^{-1}(A)=1.
\end{eqnarray*}
In other words, we have  $R_q(z;f,A)=(z+f(A))^{-1}$ for each $z>0,$
and then for each
$z\in \Sigma_\gamma.$
Hence, in particular, $(z+f(A))^{-1}$ satisfies \eqref{familyT}.

Thus, from (\ref{Ra}) and  (\ref{familyT}) it follows that
$f(A)\in \Sect(\pi-\gamma)$, where $\pi-\gamma \in
(\pi/q,\pi).$ Since $q$ can be made arbitrarily close to
$\pi/\alpha,$ so that $\gamma$ is arbitrarily close to
$\pi-\alpha,$  we conclude that $f(A)\in
\Sect(\alpha)$.
\end{proof}

Next we obtain a corollary of Theorem \ref{mainT} for certain functions from $\mathcal{NP}_+$
arising in the study of convex power series of Ritt operators.
Let
\begin{equation}\label{hc}
h(\lambda):=\sum_{n=0}^\infty c_n \lambda^n,\qquad \lambda\in \D,
\quad c_n\ge 0, \quad \sum_{n=0}^\infty c_n=1,
\end{equation}
and
\begin{equation}\label{hc1}
\mathbf{h}(\lambda):=1-h\left(\frac{1-\lambda}{1+\lambda}\right)
=1-\sum_{n=0}^\infty
c_n\left(\frac{1-\lambda}{1+\lambda}\right)^n,\qquad \lambda\in
\overline{\C}_{+}.
\end{equation}

The next result is a direct corollary of Theorem \ref{mainT}.
\begin{thm}\label{mainT2}
Let $\mathbf{h}$ be given by  (\ref{hc1}), and let $A \in
{\mathcal L}(X)$ be such that $A\in \Sect(\alpha)$, $\alpha\in
[0,\pi/2),$ and $\ker{A}=\{0\}$. Then $\mathbf{h}(A)$ is defined
in the extended holomorphic functional calculus,  and  $\mathbf{h}(A)\in \Sect(\alpha).$
Moreover,  for all
$q\in \left(2,\pi/\alpha\right)$ and $\gamma \in \left(0, \pi\left(1-\frac{1}{q}\right)\right)$,
one has
\begin{equation}\label{family2}
\|(z + \mathbf{h}(A))^{-1}\|\le \frac{c_{q,\gamma}}{|z|},\qquad z\in \Sigma_\gamma,
\end{equation}
where
\begin{equation}\label{constant2}
c_{q,\gamma}=
\frac{q M_q(A)}{2b^2\cos(\pi/q)\cos^2(\pi/q+\gamma)/2}
+\frac{2}{\cos(\pi/q+\gamma)/2},
\end{equation}
with $b=b_{q,\|A\|}=\frac{\cos(\pi/q)}{1+4\|A\|^2},$ and $M_q(A)$ given by  \eqref{epapprox}.
\end{thm}

\begin{proof}
By Lemma \ref{CorD}, $\mathbf{h}$ belongs to
$\mathcal{D}_{\pi/2}(0,1)\cap \mathcal{NP_+}.$
Moreover  $\mathbf{h}$ satisfies Definition \ref{classd} of $\mathcal{D}_{\pi/2}(0,1)$ with
\[
b=b(\beta,R)=\frac{\cos\beta}{1+R^2}\in \left(0,\min\{1,1/(2R)\}\right)
\]
and $m(\beta)=\pi/2.$ Therefore, by Theorem
\ref{mainT} (with $\theta_1=\theta_2=\pi/2$), we get
\eqref{family2} and \eqref{constant2}.
\end{proof}

We proceed with obtaining a counterpart of the preceding result
for complete Bernstein functions of sectorial operators.
It will be needed in the next section in the study of improving properties of Hausdorff functions.
As above, the explicit constants will be given since they will be
crucial for the sequel.

\begin{thm}\label{MT10}
Suppose  $\psi\in \mathcal{CBF}$ and
\eqref{ConM0101} holds for some $\omega\in (0,\pi/2)$.
Let $A\in \mathcal{L}(X)$ be such that
$A\in \Sect(\pi/2)$ and
$\ker{A}=\{0\}$.
Then $\psi(A)\in\Sect(\omega)$.
If  the numbers $\omega_0, \theta$ and $\theta_0$ are as in Proposition \ref{pr7.6},
 then for all
$q\in \left(\pi/\theta,2\right)$ and
$\gamma \in \left(0, \pi\left(1-\frac{\theta_0}{q\theta}\right)\right)$,
one has
\begin{equation}\label{family1}
\|(z + \psi(A))^{-1}\|\le \frac{c_{q,\gamma}}{|z|},\qquad z\in \Sigma_\gamma,
\end{equation}
where
\begin{equation}\label{constant3}
c_{q,\gamma}:=
\frac{2 q M_q(A)\tan(\pi/(2q))}{C\pi \cos^2(\pi/q+\gamma)/2}
+\frac{2}{\cos(\pi/q+\gamma)/2}.
\end{equation}
 with $C:=[\cos (\pi^2/(2\theta q))]^{2\theta_0/\pi}$ and $M_q(A)$ given by
  \eqref{epapprox}.
\end{thm}

\begin{proof}
Fix $\theta\in (\pi/2,\omega_0).$ Let $\omega_0$ and $\theta_0$ be defined as
in Proposition \ref{pr7.6}.
Then
\[
\psi\in \mathcal{NP_+}(\theta,\theta_0).
\]
Moreover,  from Lemma \ref{phi} it follows that $\psi \in {\mathcal D}_{\theta}(0,a)$ for each $a >0$
with
\[
m(\beta)=2\tan(\beta/2) \quad \text{and} \quad b=b(\beta,R)=1, \quad \beta \in (0,\pi/2), \quad R>0.
\]
Then, by Theorem \ref{mainT},
for $\theta_1=\theta$, $\theta_2=\theta_0$ and $\alpha=\pi/2$,
$\psi(A)$ is defined
in the extended holomorphic functional calculus,  and
\begin{equation}\label{theta0}
\psi(A)\in \Sect(\omega(\theta)),\qquad \omega(\theta)=\frac{\theta_0}{\theta}\cdot\frac{\pi}{2}.
\end{equation}
Moreover, \eqref{family1} holds with the constant given by \eqref{constant3}.

Finally, from
 (\ref{7.14}) it follows that
\[
\lim_{\theta\to \pi/2}\,\theta_0(\theta)=\omega,
\]
hence, according to \eqref{theta0},
$\lim_{\theta\to \pi/2}\,\omega(\theta)=\omega$.
So, by the definition of a sectorial operator, $\psi(A)\in \Sect(\omega)$.
\end{proof}

We turn to the proof of the main
result of this paper on convex power series of Ritt operators.
To this aim, it
will  be convenient to separate the next simple technical
statement  as a lemma.

\begin{lemma}\label{norma1}
Let $g_\epsilon,\, \epsilon\ge 0,$  be given by
$$
g_\epsilon(\lambda)=\frac{(2-\epsilon)\lambda-\epsilon}{2+\epsilon+\epsilon
\lambda}, \qquad |\lambda| \le 1.
$$
Then for every $\epsilon \ge 0$ one has $g_\epsilon \in \Wip (\D)$
and $\|g_\epsilon\|_{\Wip (\D)}=1.$
\end{lemma}
\begin{proof}
Note that
\begin{align*}
g_\epsilon(\lambda)=&-1+\frac{2(1+\lambda)}{2+\epsilon+\epsilon \lambda}\\
=&-1+\frac{2}{2+\epsilon}+\frac{4\lambda}{(2+\epsilon)(2+\epsilon+\epsilon
\lambda)}\\
=&-\frac{\epsilon}{2+\epsilon}
+\frac{4\lambda}{(2+\epsilon)^2}\sum_{n=0}^{\infty}\frac{(-1)^n
\epsilon^n \lambda^n}{(2+\epsilon)^n},\qquad |\lambda|\le 1.
\end{align*}
Hence
\begin{align*}
\|g_\epsilon\|_{\Wip (\D)}=
&\frac{\epsilon}{2+\epsilon}+\frac{4}{(2+\epsilon)^2}\sum_{n=0}^{\infty}\frac{\epsilon^n}{(2+\epsilon)^n}\\
=&\frac{\epsilon}{2+\epsilon}+\frac{2}{2+\epsilon}=1.
\end{align*}
\end{proof}

Now we are ready to prove our main result. Besides saying that a
convex power series of a Ritt operator is Ritt, it shows that a
convex power series preserves Stolz type. Moreover we have a
control over the angle of the Ritt operator given by the series.

In the course of our proof we first show that $1-
h(T)=\mathbf{h}(\mathcal C(T))$ in a ``simple'' case when $T$ is
Ritt and $1-T$ is invertible. Then the sectoriality estimate for
$\mathbf{h}(\mathcal C(T))$ given by Theorem \ref{mainT1}
transfers to that for $1-h(T).$ The general case then follows by
approximation, and the uniformity of the constant in
\eqref{family} with respect to a family approximating $T$ appears
to be indispensable.
\begin{thm}\label{mainT1}
Let  $h$ be defined  by (\ref{hc}), and let $T$ be a Ritt operator
on $X$ . Then there exists $\omega\in [0,\pi/2)$ such that
${\mathcal C}(T) \in \Sect(\omega),$ and $h(T)$ is a  Ritt
operator on $X$ with the same angle $\omega.$ Moreover, if $T$ is
of Stolz type $\sigma,$ then $h(T)$ has Stolz type $\sigma$ as
well.
\end{thm}

\begin{proof}
Since by assumption $T$ is Ritt, $T$ is power bounded and $\sigma(T)\subset \D \cup\{1\}.$

Assume first that $1 \not \in \sigma(T),$ so that $\sigma(T)\subset \D.$
Let $\mathcal{C}$ be the Cayley transform defined by \eqref{funU},
and set for shorthand $A:=\mathcal{C}(T).$ Clearly $A \in {\mathcal L}(X),$ and
Proposition \ref{Stolz} implies that $A \in \Sect(\omega)$ for
some $\omega\in [0,\pi/2).$ Moreover, $\ker (A)=\{0\}.$ Let us
first prove that
\begin{equation}\label{SA}
1- h(T)=\mathbf{h}(A),
\end{equation}
where $\mathbf{h}(A)$ is defined by the extended holomorphic
calculus and $h(T)$ is given by the $\Wip(\D)$-calculus.

By the definition of the extended holomorphic calculus,
\begin{equation}\label{cchh}
\mathbf{h} (A)=A^{-1}(1+A)^2(\mathbf{h} \cdot \tau)(A).
\end{equation}
If $\omega' \in (\omega,\pi/2)$, then
using (\ref{ATeq}) and Cauchy's theorem, we obtain
\begin{eqnarray*}
&&(\mathbf{h} \cdot \tau)(A)=\frac{1}{2\pi i}
\int_{\partial \Sigma_{\omega'}}
\frac{\lambda\,\mathbf{h} (\lambda)}{(\lambda+1)^2}(\lambda-A)^{-1}\,d\lambda\\
&=&\frac{1}{2\pi i}
\int_{\partial \Sigma_{\omega'}}
\frac{\lambda}{(\lambda+1)^2}
\sum_{n=0}^\infty c_n\left[1-\left(\frac{1-\lambda}{1+\lambda}\right)^n\right]
(\lambda-A)^{-1}\,d\lambda\\
&=&\frac{1}{2\pi i}
\sum_{n=0}^\infty c_n\int_{\partial \Sigma_{\omega'}}
\frac{\lambda}{(\lambda+1)^2}
\left[1-\left(\frac{1-\lambda}{1+\lambda}\right)^n\right]
(\lambda-A)^{-1}\,d\lambda\\
&=&\sum_{n=0}^\infty c_n
A(A+1)^{-2}\,[1-((1-A)(1+A)^{-1})^n]\\
&=&A(A+1)^{-2}\sum_{n=0}^\infty c_n
(1-T^n)\\
&=&A(A+1)^{-2}(1-h(T)).
\end{eqnarray*}
This and (\ref{cchh}) imply (\ref{SA}).

From (\ref{SA}) and Theorem \ref{mainT2} it follows that $1-h(T)\in
\Sect(\omega)$, where $\omega\in [0,\pi/2)$ is a sectoriality
angle of $A={\mathcal C}(T)$. Hence, by the spectral mapping
theorem (Theorem \ref{connect}) and Theorem \ref{sectorritt}, $h(T)$ is a Ritt operator
of angle $\omega.$ Thus the first statement of the theorem is proved if $T$ is Ritt
such that $1-T$ is invertible.

Let now $1 \in \sigma(T).$
Then
 consider the approximation family $(T_\epsilon)_{\epsilon\in
 (0,1)}\subset
 \mathcal{L}(X)$ given by
\begin{equation}\label{epsA}
T_\epsilon:=g_\epsilon(T)=((2-\epsilon)T-\epsilon)(2+\epsilon+\epsilon T)^{-1},\qquad \epsilon\in (0,1).
\end{equation}
Observe  that since $T$ is Ritt,
 the spectral mapping theorem for the (standard)
Riesz-Dunford functional calculus and simple geometric
considerations  imply that $$\sigma(T_\epsilon)\subset
\D.$$ Thus $h(T_\epsilon)$ is well-defined in the
$\Wip(\D)$-calculus.

Furthermore, note that
for every $\epsilon \in (0,1),$
\begin{equation}\label{CT}
\mathcal{C}(T_\epsilon)=(1+\epsilon -T+\epsilon T)(1+T)^{-1}=\mathcal{C}(T)+\epsilon,
\end{equation}
so that $\mathcal{C}(T_\epsilon) \in {\rm Sect}(\omega).$
Moreover,
\begin{equation}\label{eps0}
\lim_{\epsilon\to
0}\,\|T-T_\epsilon\|=\|\epsilon(1+T)^2(2+\epsilon+\epsilon
T)^{-1}\|=0.
\end{equation}
Hence for every $n\in \Z_+,$
\begin{equation}
\lim_{\epsilon \to 0} \|T^n-T^n_\epsilon \|=0.
\end{equation}
Since, in view of Lemma \ref{norma1},
$$\|h(T_\epsilon)\|=\|h (g_\epsilon(T))\|_{\Wip(\D)} \le
\sum_{n=0}^{\infty} c_n \|g_\epsilon(T)\|^n_{\Wip(\D)}=M h(1)=M,$$
where $M:=\sup_{n \ge 0} \|T^n\|,$
the dominated convergence theorem gives
\begin{equation}\label{approx}
\lim_{\epsilon \to 0}\|h(T)-h(T_\epsilon)\|=0.
\end{equation}
(Note that we do need any composition rule here.)

According to the first part of the proof and \eqref{CT}, for each $\epsilon \in (0,1),$
\begin{equation}\label{equality}
1-h(T_\epsilon)=\mathbf{h}(\mathcal{C}(T_\epsilon))=\mathbf{h}(\mathcal{C}(T)+\epsilon)=\mathbf{h}(A+\epsilon).
\end{equation}
Next we use Theorem \ref{mainT2} again. The estimate given by
\eqref{family} and \eqref{constant} and Proposition \ref{PrBBL}
yield
 \begin{equation}\label{spectrum}
  -\sigma (\mathbf{h}(A+\epsilon)) \subset
\mathbb C \setminus \Sigma_\omega \end{equation}
 and
\begin{equation}\label{uniform}
 \|(z+\mathbf{h}(A+\epsilon))^{-1}\| \le
\frac{M_{\omega'}}{|z|}, \qquad z \in \Sigma_{\omega'},
\end{equation}
for all $\epsilon \in (0,1)$ and $\omega'>\omega,$ where
$M_{\omega'}$ \emph{does not depend} on $\epsilon.$
 Now due to
\eqref{spectrum} and \eqref{equality} we have
$\sigma(h(T_\epsilon)) \subset 1-\overline{\Sigma}_\omega$, so
that by \eqref{approx},
$$
\sigma(h(T)) \subset 1-\overline{\Sigma}_\omega.
$$
Moreover, \eqref{equality} and \eqref{uniform} imply
\begin{equation}\label{uniform1}
\|(z-1+h(T_\epsilon))^{-1}\| \le \frac{M_{\omega'}}{|z|}, \qquad z
\in \Sigma_{\omega'},
\end{equation}
for all $\epsilon \in (0,1)$ and $\omega'>\omega.$ Then, by
\eqref{approx} and \eqref{uniform1},
\begin{equation*}
 \lim_{\epsilon \to 0} (z-1+h(T_\epsilon))^{-1}=(z-1+h(T))^{-1}
\end{equation*}
strongly in ${\mathcal L}(X)$ for each $z \in \Sigma_\omega.$
Therefore, for all $\omega'>\omega$ and $z \in \Sigma_{\omega'},$
$$
\|(z-1+h(T))^{-1}\| \le \liminf_{\epsilon \to 0}
\|(z-1+h(T_\epsilon))^{-1}\| \le  \frac{M_{\omega'}}{|z|}.
$$
In other words, $h(T)$ is a Ritt operator of angle $\omega$.

Let us now prove the claim about Stolz type. Assume that $T$ is a
Ritt operator of Stolz type $\sigma.$ Then by the preceding part
of the proof and Proposition \ref{Stangle} the operator $h(T)$ is
Ritt of angle $\alpha=\arccos(1/\sigma)$. Hence for any $\beta\in
(\alpha,\pi),$
\begin{equation}\label{ApAr}
\|(z-h(T))^{-1}\|\le \frac{C_\beta}{|z-1|},\qquad
1-z\not\in \overline{\Sigma}_\beta.
\end{equation}

On the other hand, by Theorem \ref{connect} and
Proposition \ref{Stangle1}, we conclude that
\begin{equation}\label{SpM}
\sigma(h(T))\subset \overline{S}_\sigma.
\end{equation}
Let $\delta>\sigma$ be fixed, and let $
\sigma_0:=(\sigma+\delta)/2$ and $\alpha_0:=\arccos(1/\sigma_0).$
Then, by \eqref{ApAr} and \eqref{stolsec}, there is $C_{\alpha_0}\ge 1$ such that
\[
\|(z-h(T))^{-1}\|\le \frac{C_{\alpha_0}}{|z-1|}, \qquad 1-z\not\in \overline {\Sigma}_{\alpha_0}.
\]
If $z \in \D \setminus S_{\delta}$ and $1-z \in
\overline{\Sigma}_{\alpha_0},$ then a simple calculation shows that
\[{\rm Re} (1-z) \ge \frac{2\delta(\delta-\sigma_0)}{(\delta^2-1)\sigma_0^2}.\]
Therefore, the distance between $(\D \setminus S_{\delta})\cap(1-\overline{\Sigma}_{\alpha_0})$ and $S_\sigma$ is positive, and,
in view of \eqref{SpM},
\[
\|(z-h(T))^{-1}\|\le {\tilde C}, \qquad
z \not\in {S_{\delta}},\quad 1-z \in
\overline{\Sigma}_{\alpha_0},
\]
for some ${\tilde C}> 0.$
Taking into account
\[
\C\setminus S_\delta\subset \{z \in \C:\,1-z \not\in
\overline {\Sigma}_{\alpha_0}\}\cup \{z  \in \C:
\,z\not\in {S}_\delta,\quad 1-z \in
\overline{\Sigma}_{\alpha_0}\},
\]
we conclude that the operator $h(T)$ satisfies  \eqref{dopE}. As the choice of
$\delta>\sigma$ is arbitrary, $h(T)$ is of Stolz type $\sigma.$
\end{proof}

\section{Hausdorff functions of Ritt operators: improving properties}\label{improvesec}

As we mentioned in the introduction, our technique allows one  to
characterize improving properties of certain $\Wip(\D)$-functions,
namely of its subclass consisting of Hausdorff functions. In
particular, the following simple geometric criterion holds.
\begin{thm}\label{Dis}
Let $h$ be a non-constant regular Hausdorff function, and let
 $\gamma\in (0,\pi/2)$ be fixed. The following statements are
 equivalent.
 \begin{itemize}
\item [(i)] One has
\begin{equation}\label{ConM}
1-h(\lambda)\subset \overline{\Sigma}_\gamma,\qquad \lambda\in \D.
\end{equation}
\item [(ii)] For every Banach space $X$ and every power
bounded operator $T$ on $X$  the operator $h(T)$ is Ritt  of angle
$\gamma$.
\end{itemize}
\end{thm}

\begin{proof}

Let us first prove that (i) implies (ii). Let
$\psi$ be defined by
\begin{equation}\label{ConM1}
\psi(\lambda):=1-h(1-\lambda), \qquad \lambda \in \D,
\end{equation}
and denote its extension to $\mathbb C\setminus (-\infty,0]$ given by Proposition
\ref{prop1} by the same symbol. Thus, $\psi \in \mathcal{CBF}$
and by assumption
\begin{equation}\label{an}
\psi(\lambda)\in \overline{\Sigma}_\gamma,\qquad \lambda\in \D_1=:1+\D=\{\lambda\in
\C:\,|\lambda-1|< 1\}.
\end{equation}
As for Theorem \ref{mainT1}, the proof will be done in two steps. Suppose first that
\begin{equation}\label{inclus}
\{-1, 1\} \not \subset \sigma(T).
\end{equation}

If $ \varphi(\lambda):=\frac{2\lambda}{\lambda+1}, $ then $\varphi \in
\mathcal{CBF},$ and $\psi\circ\varphi \in \mathcal{CBF}$ as a
composition of complete Bernstein functions. Moreover, since
$\varphi:\mathbb C_+ \to \D_1,$  it follows from \eqref{an} that
\[
(\psi\circ\varphi)(\C_{+})\subset \overline{\Sigma}_\gamma.
\]

Noting that
\[
\varphi^{-1}(\lambda)=\frac{\lambda}{2-\lambda},\quad \lambda\not=2,
\]
and setting  $Q:=1-T$, we conclude by Proposition \ref{Lm} that
\[
\varphi^{-1}(Q)=Q(2-Q)^{-1}=\mathcal C (T)\in \Sect(\pi/2),
\]
and, by Theorem \ref{GTmain}, b)  and the composition rule \eqref{comprule}, we
obtain that
\[
\psi(Q)=(\psi\circ\varphi)(\varphi^{-1}(Q))=(\psi\circ\varphi)({\mathcal
C}(T)) \in \Sect(\gamma).
\]
Furthermore, by Lemma \ref{HaTol},
\begin{equation}\label{equalit}
\psi(Q)=1-h(T),
\end{equation}
where $\psi(Q)$ is defined in the extended holomorphic functional calculus,
and $h(T)$ is given by $\Wip(\D)$-calculus.

Observe that by \eqref{disc}, $h(\overline \D)\subset \overline \D.$
Moreover, if $\lambda \in \overline \D$ and $|h(\lambda)|=1,$ then
\[
1=\left|\sum_{k=0}^{\infty} c_k \lambda^k \right|\le \sum_{k=0}^{\infty} c_k=1,
\]
and $\lambda=h(\lambda)=1,$ so that $h(\overline \D)\subset \D \cup\{1\}.$
From  Theorem \ref{connect} it then follows that
$\sigma(h(T))\subset \D\cup\{1\}$, and  by Theorem \ref{MT10} we conclude that
$h(T)$ is a
Ritt operator of angle $\gamma$. Thus the statement is proved
for power bounded $T$ such that \eqref{inclus} holds.

If  \eqref{inclus} does not hold, then we let $\sup_{n \ge 0} \|T^n\|:=M$ and consider the family of bounded
linear operators $T_\epsilon:=(1-\epsilon)T,
\epsilon\in (0,1). $ Clearly, $\sigma(T_\epsilon) \subset
(1-\epsilon)\overline \D$ for each $\epsilon \in (0,1)$
 and
\begin{equation}\label{el1}
\|(z-T_\epsilon)^{-1}\|=(1-\epsilon)^{-1}
\|((1-\epsilon)^{-1}z-T)^{-1}\|
\le \frac{M}{|z|-1},\qquad |z|>1.
\end{equation}
Hence, by the first step, if $Q_\epsilon=1-T_\epsilon$ then
$\psi(Q_\epsilon)\in \Sect(\gamma)$.
Moreover, as  $\|T_\epsilon\|\le \|T\|$ and (\ref{el1}) holds,
by (\ref{A0}) and Theorem \ref{MT10} again we infer that for each $\beta\in (\gamma,\pi)$
there exists  $M_\beta$  \emph{independent} of $\epsilon>0$ such that
\begin{equation}\label{zv}
\|(z+\psi(Q_\epsilon))^{-1}\|\le \frac{M_\beta}{|z|},\qquad
z \in \Sigma_{\pi-\beta}.
\end{equation}
Thus, taking into account  \eqref{equalit}, we infer that
\begin{equation}\label{zv1}
\|(1-h(T_\epsilon)+z)^{-1}\|\le \frac{M_\beta}{|z|},\qquad
z \in \Sigma_{\pi-\beta},
\end{equation}
for each $\beta \in (\gamma,\pi).$
Moreover, since $h \in \Wip(\D),$ we have
\begin{equation}\label{hhh}
\lim_{\epsilon \to 0}\|h(T_\epsilon)-h(T)\|=0,
\end{equation}
hence Proposition \ref{connect} and \eqref{kato} yield
\begin{equation}\label{spect}
\sigma(h(T))\subset 1 -\overline{\Sigma}_\gamma.
\end{equation}
Now, combining \eqref{zv1} and \eqref{hhh}, we infer that
\begin{equation}
(1-h(T) +z)^{-1}=\lim_{\epsilon \to 0
}(1-h(T_\epsilon)+z)^{-1}
\end{equation}
strongly in ${\mathcal L}(X)$ for every nonzero $z \in \mathbb C \setminus \overline{\Sigma}_{\pi-\gamma}.$
Therefore,
\begin{equation}\label{resolv}
\|(1-h(T) +z)^{-1}\|\le \liminf_{\epsilon \to 0
}\|(1-h(T_\epsilon)+z)^{-1}\| \le \frac{M_\beta}{|z|}, \qquad z \in \Sigma_{\pi-\beta},
\end{equation}
for every $\beta \in (\gamma,\pi).$ Now, \eqref{spect} and
\eqref{resolv} imply  the claim.

 The implication (ii) $\Rightarrow$ (i) is proved in
  \cite[p. 1728]{Dungey}. It suffices to consider the
  multiplication operator $(T f)(\lambda)=\lambda f(\lambda), \lambda \in \D,$ on $X =C(\overline{\D})$
  and to use the fact that if $h(T)$ is Ritt of angle $\omega,$ then the
  multiplication semigroup $(e^{(1-h(T))t})_{t \ge 0}:$
$$
(e^{(1-h(T))t} f)(\lambda)=e^{(1-h(\lambda))t}f(\lambda), \qquad \lambda \in \overline
{\D},
$$
  is sectorially bounded on $\Sigma_{\omega'}$ for every $\omega' >\omega$.
\end{proof}

As an illustration of Theorem \ref{Dis},  we show how several main results
 from \cite{Dungey} can be obtained by  our technique and answer a question
posed in \cite{Dungey}. Moreover, we  show that Theorem \ref{Dis}
provides ``geometrical'' improvements of the results from
\cite{Dungey}.

\begin{example}\label{rem1}
By Example \ref{discex}, a) the function
$h_\alpha(\lambda)=1-(1-\lambda)^\alpha$ is regular Hausdorff for every
$\alpha \in (0,1).$ Moreover,
$$(1-h_\alpha)(\D) \subset \overline{\Sigma}_{\alpha \pi/2}.$$ Thus,
for any power-bounded operator $T$ on $X$ and any $\alpha\in
(0,1),$ the operator $h_\alpha (T)$ is Ritt of angle $\alpha
\pi/2.$ Clearly,
 Theorem \ref{Dis}  extends
 \cite[Theorem 1.1 and Theorem 4.3]{Dungey}, where the special case of
Theorem \ref{Dis} for $h_\alpha, \alpha \in (0,1),$ was
considered.
\end{example}

\begin{example}\label{zeta}
For $\alpha>0$ define
\[
L_{1+\alpha}(\lambda):=\frac{1}{\zeta(1+\alpha)}\sum_{n=1}^\infty
\frac{\lambda^n}{n^{1+\alpha}},
\]
where $\zeta$ is the zeta function. The functions $L_{1+\alpha}$
arise as generating functions for so-called zeta-probabilities
\cite{Dungey}. They are also related to fractional polylogarithms,
see e.g. \cite{Costin}.  As discussed in \cite{Dungey}, they
appear to be useful in probability theory.

Clearly, $L_{1+\alpha} \in \Wip(\D)$ for every $\alpha >0.$ It was
proved in \cite[Theorem 1.2 and Theorem 4.4]{Dungey} that for
every $\alpha\in (0,1)$ and any power-bounded operator $T$ on $X$
the operator $L_{1+\alpha}(T)$ is Ritt. We show that this result
follows from Theorem \ref{Dis} and, moreover, we are able to
provide a bound for the corresponding angle of $L_\alpha(T)$. To
this aim, note that since for every $k\in \N,$
\[
\frac{1}{k^{1+\alpha}}=\frac{1}{\Gamma(1+\alpha)}\int_0^\infty e^{-kt}t^{\alpha}\,dt
=\frac{1}{\Gamma(1+\alpha)}\int_0^1 (\log (1/s))^\alpha s^{k-1}\,ds,
\]
the function $L_{1+\alpha}$ is  Hausdorff for each $\alpha \in
(0,1)$. Moreover, by \cite[p 29, 1.11 (8)]{Bateman2} we have
\begin{align}\label{FAs}
\zeta(1+\alpha)L_{1+\alpha}(\lambda)=&\lambda\Phi(\lambda,1+\alpha;1)\\
=&\zeta(1+\alpha)+\Gamma(-\alpha)(\log(1/\lambda))^\alpha+
\mbox{O}(|\lambda-1|), \notag
\end{align}
as $\lambda\to 1, \lambda \in \overline{\D},$ where $\Phi$ is the Lerch zeta
function (see e.g. \cite{Ap}). Taking into account the inequality
\[
|L_{1+\alpha}(\lambda)|\le 1,\qquad \lambda \in \overline{\D},
\]
and (\ref{FAs}), we infer that
\[
1-L_{1+\alpha}(\lambda)\in \overline{\Sigma}_\beta,\qquad \lambda \in
\overline{\D},
\]
for some $\beta=\beta(\alpha)\in (0,\pi/2)$. Hence, by Theorem
\ref{Dis}, for any power-bounded operator $T$ on $X$  and any
$\alpha\in (0,1)$ the operator $L_{1+\alpha}(T)$ is Ritt of angle
$\beta(\alpha)$.

However, we can be more precise here. By combining Theorem
\ref{Dis} with Proposition \ref{Lfun} from Appendix B, we conclude
that $L_{1+\alpha}(T)$ is of angle $\beta(\alpha)=\alpha\pi/2.$
\end{example}

\begin{example}\label{eps}
Let for a fixed $\epsilon\in (0,1),$
\begin{equation}\label{int1}
h_\epsilon(\lambda):=1-\frac{1}{\epsilon}
\int_0^\epsilon (1-\lambda)^\alpha\,d\alpha=1-\frac{(1-\lambda)^\epsilon-1}{\epsilon\log(1-\lambda)},\qquad \lambda\in \D.
\end{equation}
The function $h_\epsilon$ extends holomorphically to
$\C\setminus (-\infty,0],$ and denoting the extension by the same
symbol we infer that
\[
f_\epsilon(\lambda):=1-h_\epsilon(1-\lambda)=\frac{\lambda^\epsilon-1}{\log
\lambda^\epsilon},\qquad \lambda > 0,
\]
belongs to $\mathcal{CBF}$ as a composition of the complete
Bernstein functions $(\lambda-1)/\log \lambda$ and $\lambda^\epsilon.$ If $\lambda\in
\C_{+}$, then $\lambda^\epsilon \in \overline{\Sigma}_{\epsilon\pi/2},$
hence in view of the integral representation in \eqref{int1},
$
f_\epsilon(\C_{+})\in \overline{\Sigma}_{\epsilon\pi/2}.
$
Thus, by Example  \ref{discex}, b), the function
$h_\epsilon$ is regular Hausdorff and
$$(1-h_\epsilon)(\D) \subset \overline{\Sigma}_{\epsilon\pi/2}.$$

By  Theorem \ref{Dis}, we conclude that if $T$ is a power-bounded
operator on $X$, then $h_\epsilon(T)$ is Ritt of angle
$\epsilon\pi/2$. This settles a conjecture posed in \cite[p.
1735]{Dungey}. Note that Theorem 5.1 in \cite{Dungey} can also be
treated in a similar way, but we leave the details to the interested
reader.

Note that the angle $\epsilon\pi/2$ given by Theorem \ref{Dis} is
not optimal. For instance, for a Hausdorff function $h_1,$
\[
h_1(\lambda)=1-f_1(1-\lambda)=1+\frac{\lambda}{\log(1-\lambda)},\qquad \lambda\in \D,
\]
we have by Lemma \ref{QBC} from Appendix B:
\[
(1-h_1)(\D)\subset \overline{\Sigma}_{\pi/3}.
\]
So, by Theorem \ref{Dis}, we obtain that for any power-bounded
operator $T$ on $X$ the operator $h_1(T)$ is Ritt of angle
$\pi/3$.

Remark, finally, that \cite{Dungey} deals with the elements of
$\ell_1(\Z_+)$
 given by $\frac{1}{\epsilon}
\int_0^\epsilon A_\alpha \,d\alpha,$ where $A_\alpha \in
\ell_1(\Z_+)$  is the sequence of Taylor coefficients
of $(1-z)^\alpha,$ rather than the generating function of
$\frac{1}{\epsilon}
\int_0^\epsilon A_\alpha \,d\alpha.$ However, since $\Wip(\D)$ and $\ell_1(\Z_+)$ are isometrically
isomorphic as Banach algebras, both settings are, in fact, equivalent.
\end{example}

%%%%%%%%%%%%%%%%%%%%%%%%%%%%%%%%%%%%%%%%%%%%%%
\section{A remark on angles of Ritt operators}
%%%%%%%%%%%%%%%%%%%%%%%%%%%%%%%%%%%%%%%%%%%%%%

In this section, by means of a simple example, we illustrate the
statement of  Theorem \ref{mainT1} on angles of Ritt operators. To this aim, we introduce the following notation.
Let the \emph{minimal} angle $\alpha(T)$ of a Ritt operator $T$ on $X$ be defined as
$$
\alpha(T)=\inf\{\alpha: T \, \text{is Ritt of angle} \, \alpha\}.
$$
 We
show that there exists a Ritt operator  $T$ with the minimal angle $\alpha(T)$
and a function $h$ satisfying \eqref{hc} such that $h(T)$ is a
Ritt operator with the minimal angle $\alpha(h(T))$ greater than $\alpha(T)$. Moreover, the difference $\alpha(h(T))-\alpha(T)$ can be arbitrarily close to $\pi/2.$ Thus
discrete subordination does not, in general, preserve angles of
Ritt operators. This justifies, in particular, the use of the Cayley transform and Stolz types in
the study of permanence properties for discrete subordination.

In the notation of Theorem \ref{mainT1}, let
\[
h(\lambda):=\lambda^2,\qquad \lambda \in \D,
\]
and for $\varphi\in(0,\pi/2)$ and  $\rho\in (0,2\cos\varphi)$ set $ \lambda=1-\rho e^{i\varphi}.$
Note that
\begin{align}\label{So}
\frac{|{\rm Im}\, (1-h(\lambda))|}{\tan\varphi\,{\rm Re}\, (1-h(\lambda))}
=&\frac{|2\sin\varphi-\rho\sin(2\varphi)|}
{\tan\varphi \,(2\cos\varphi-\rho\cos (2\varphi))}\\
=&\frac{|1-2t\cos^2\varphi|} {1-t\cos (2\varphi)},\notag
\end{align}
where
$t=\rho/(2\cos\varphi)$ so that $t \in (0,1).$

On the other hand,
\begin{eqnarray}\label{cab}
\frac{|{\rm Im}\,\mathcal{C}(\lambda)|}{\tan\varphi\,{\rm Re}\,\mathcal{C}(\lambda)}=
\frac{2\cos\varphi}{2\cos\varphi-\rho}=
\frac{1}{1-t}.
\end{eqnarray}

Let $X=L^2((0,1)).$ For fixed $\varphi\in (0,\pi/2)$ and $\delta\in
(0,1)$ define the operator $T_{\varphi,\delta}$ on $X$ by
\begin{equation}\label{Fu}
(T_{\varphi,\delta}y)(t):= (1-2t\delta\cos\varphi
e^{i\varphi})y(t),\qquad  y\in L^2((0,1)).
\end{equation}

\begin{thm}\label{ShT}
Let $X$ and $T_{\varphi,\delta} \in {\mathcal L}(X)$  be defined by \eqref{Fu}. Then
\begin{itemize}
\item [(i)] $T_{\varphi,\delta}$ is a Ritt operator of the minimal angle $\alpha(T_{\varphi,\delta})=\varphi.$
\item [(ii)] for each $\epsilon\in (0,1)$ there exist
 $\varphi\in (0,\pi/2)$ and $\delta\in (0,1)$ such that
$T_{\varphi,\delta}^2$ is a Ritt operator of the minimal angle
$\alpha(T_{\varphi,\delta}^2):=\beta_{\varphi,\delta}$ satisfying
\begin{equation}\label{MtT}
\tan\beta_{\varphi,\delta}\ge \frac{\tan\varphi}{\epsilon}.
\end{equation}
Hence, $\beta_{\varphi,\delta} - \tan\varphi$ can be arbitrarily close to
$\pi/2.$
\item [(iii)] If $\gamma_{\varphi,\delta}$ is the minimal sectoriality angle of ${\mathcal
C}(T_{\varphi,\delta}),$ then $\gamma_{\varphi,\delta}$ can be arbitrarily
close to $\beta_{\varphi,\delta}.$
\end{itemize}
\end{thm}
\begin{remark}
Recall that $T_{\varphi,\delta}^2$ is Ritt of angle
$\gamma_{\varphi,\delta}$ by Theorem \ref{mainT1}.
\end{remark}

\begin{proof}

A direct  calculation shows that $T_{\varphi,\delta}$ is a Ritt
operator of the minimal angle $\varphi,$ and thus proves (i). Hence,
by \eqref{cab},
\begin{equation}\label{alpgam}
\tan\gamma_{\varphi,\delta}=\frac{\tan\varphi}{1-\delta}.
\end{equation}
So, $\gamma_{\varphi,\delta}>\varphi.$

Let
\[
v_\varphi(t):=\frac{|1-2t\cos^2\varphi|} {1-t\cos (2\varphi)},
\]
where $t=\rho/(2\cos\varphi).$
Observe that
by \eqref{So},
\[
\tan\beta_{\varphi,\delta}=\tan\varphi\,\sup_{t\in (
0,1)}\,v_\varphi(\delta t) \ge v_\varphi(\delta)\tan\varphi.
\]
Hence from  \eqref{alpgam} it follows that
\[
\frac{\tan\beta_{\varphi,\delta}}{\tan\gamma_{\varphi,\delta}}
\ge v_\varphi(\delta)
=\frac{(1-\delta)(\delta\cos(2\varphi)+\delta-1)}{1-\delta\cos(2\varphi)}.
\]
Let $\epsilon\in (0,1)$ be fixed and let $\varphi=\arccos \delta^{\epsilon/2}/2.$
Then
\begin{align*}
\limsup_{\delta\to 1}\, \frac{\tan\beta_{\varphi,\delta}}{\tan\gamma_{\varphi,\delta}} &\ge \limsup_{\delta\to 1}\,
\frac{(1-\delta)(\delta^{1+\epsilon/2}+\delta-1)}{1-\delta^{1+\epsilon/2}}\\
 &=1/(1+\epsilon/2)\\
 &>1-\epsilon.
 \end{align*}
 This shows (iii). Now  if  $\delta\in (0,1)$ is such that
\[
\frac{\tan\beta_{\varphi,\delta}}{\tan\gamma_{\varphi,\delta}}>1-\epsilon \quad \mbox{and}\quad 1-\delta<(1-\epsilon)\epsilon,
\]
then
\begin{equation*}
\tan\beta_{\varphi,\delta}>(1-\epsilon)\tan\gamma_{\varphi,\delta}=
\frac{1-\epsilon}{1-\delta}\tan\varphi\ge \frac{\tan\varphi}{\epsilon},
\end{equation*}
and (ii) follows.
\end{proof}

\section{Appendix A}

In this appendix we collect several technical estimates used in
previous sections. The lemma below is crucial in the
characterization of Ritt operators in terms of Stolz domains given
in Proposition \ref{rittchar}.
\begin{lemma}\label{Lfir}
Let $z\in \D$ and  let
\[
Q_z(\varphi):=\frac{|z-e^{i\varphi}|}{|1-e^{i\varphi}|},\qquad \varphi\in
[0,2\pi).
\]
Then
\begin{equation}\label{Fmin}
m_z:=\min_{\varphi\in
[0,2\pi)}\,Q_z(\varphi)=\frac{1-|z|^2}{2|1-z|}.
\end{equation}
\end{lemma}

\begin{proof}
Let $z=re^{i\alpha},$ where $r\in [0,1)$ and $\alpha\in
[0,2\pi).$ Then, setting $ \psi:=\varphi/2\in (0,\pi),$ we obtain
\begin{align*}
4Q_z^2(\varphi) =&\frac{r^2+1-2r\cos\alpha
+4r\cos\alpha\sin^2\psi-4r\sin\alpha\sin\psi\cos\psi}{\sin^2\psi}\\
=&|1-z|^2(1+\cot^2\psi)+4r\cos\alpha-4r\sin\alpha \cot\psi\\
=&\left(|1-z|\cot\psi-\frac{2r\sin\alpha}{|1-z|}\right)^2
-\frac{4r^2\sin^2\alpha}{|1-z|^2}+|1-z|^2+4r\cos\alpha.
\end{align*}
Hence a simple calculation shows that
\begin{align*}
4m_z^2=&-\frac{4r^2\sin^2\alpha}{|1-z|^2}+|1-z|^2+4r\cos\alpha\\
=&\frac{(1-|z|^2)^2}{|1-z|^2},
\end{align*}
and  \eqref{Fmin} follows.
\end{proof}

The next simple lemma is instrumental in the proof of  Theorem
\ref{mainT}.

\begin{lemma}\label{sum}
 For all $\gamma\in [0,\pi)$, $\beta\in [0,\pi)$ such
that  $\gamma+\beta<\pi$,
\begin{equation}\label{FEH}
|z+\lambda|\ge \cos ((\gamma+\beta)/2)\, (|z|+|\lambda|),\quad
z\in \overline{\Sigma}_\gamma,\quad \lambda\in
\overline{\Sigma}_\beta.
\end{equation}
\end{lemma}
\begin{proof}
Note first that if $\beta\in (-\pi,\pi)$ and $s>0$ then
\begin{equation}\label{number}
|1+se^{i\beta}|^2=1+s^2+2s\cos\beta \ge\cos^2(\beta/2) (1+s)^2.
\end{equation}
Let now $ \gamma>0,\quad \beta>0,\quad \gamma+\beta<\pi, $ and
\[
z=r e^{i\gamma_0}\in \overline{\Sigma}_\gamma,\quad \lambda= \rho
e^{i\beta_0}\in \overline {\Sigma}_\beta, \quad |\gamma_0|\le
\gamma,\quad |\beta_0|\le \beta.
\]
Then, using (\ref{number}), we obtain
\[
|z+\lambda|=r|1+r^{-1}\rho e^{i(\beta_0-\gamma_0)}|\ge
\cos((\beta_0-\gamma_0)/2)\,(|z|+|\lambda|).
\]
From this, since
\[
|\beta_0-\gamma_0|\le \beta+\gamma \in [0,\pi), \quad \text{ and}
\quad \cos((\beta_0-\gamma_0)/2)\ge \cos((\beta+\gamma)/2),
\]
it follows that
\begin{equation}\label{number1}
|z+\lambda|\ge \cos((\beta+\gamma)/2)\,(|z|+|\lambda|),\quad z\in
\overline{\Sigma}_\gamma,\quad \lambda\in \overline{\Sigma}_\beta.
\end{equation}
\end{proof}

Now we turn to the proof of Lemma \ref{CorD}  which was essential
in our arguments leading to Theorems \ref{mainT2} and
\ref{mainT1}. The proof of this lemma is based on several auxiliary estimates.
\begin{lemma}\label{R}
For all $R>0$ and $\beta\in (-\pi/2,\pi/2)$ there exists
\[
b=b(\beta,R)\in (0,\min\{1,1/(2R)\}),
\]
such that
\begin{equation}\label{Est1}
\left|\frac{1-re^{i\beta}}{1+re^{i\beta}}\right|\le
\frac{1-b r}{1+b r},\qquad r\in (0,R),
\end{equation}
and
\begin{equation}\label{C0}
\frac{1}{|1+re^{i\beta}|^2}\le
\frac{1}{(1+b r)^2},\qquad r\in (0,R).
\end{equation}
Moreover, for each fixed $R>0$, $ b=b(\cdot,R)$ can be arranged to be decreasing function on $(0,\pi/2).$
\end{lemma}

\begin{proof}
The estimate (\ref{Est1}) is equivalent to
\[
(1+r^2-2r\cos\beta)(1+b r)^2\le
(1+r^2+2r\cos\beta)(1-b r)^2.
\]
Rearranging terms yields
\[
(1+r^2)[(1+b r)^2-(1-b r)^2]\le 2r\cos\beta[(1+b r)^2+(1-b r)^2],
\]
or
\[
\alpha (1+r^2)\le \cos\beta(1+b^2 r^2),\qquad r\in (0,R).
\]
The last inequality holds if
\begin{equation}\label{alpha}
b=b(\beta,R)=\cos\beta/(1+R^2)\in (0,\min\{1,1/(2R)\}).
\end{equation}
Moreover, if $b$ is defined by
(\ref{alpha}), then
\[
(1+b r)^2\le (1+r\cos\beta)^2\le
1+r^2+2r\cos\beta=|1+re^{i\beta}|^2,\qquad r>0,
\]
and (\ref{C0}) follows.

Given the definition \eqref{alpha}, the last claim is straightforward.
\end{proof}

Define for each $n \in \mathbb N$
\begin{equation}\label{qn}
h_n(\lambda):=\left(\frac{1-\lambda}{1+\lambda}\right)^n,\qquad q_n(\lambda):=1-h_n(\lambda), \qquad \lambda\not=-1.
\end{equation}

Note that $|h_n(\lambda)|\le 1$, $|q_n(\lambda)|\le 2$, $\lambda\in \overline{\C}_{+}, $ and
that $q_n$ maps $\overline{\C}_{+}$ into $\overline{\C}_{+}.$
Moreover, for every $n\in \N,$
\begin{eqnarray*}
h_n'(\lambda)&=&-2n\frac{h_{n-1}(\lambda)}{(1+\lambda)^2},\\
h_n''(\lambda)&=&\frac{4n}{(1+\lambda)^4}[(n-1)h_{n-2}(\lambda)+(1+\lambda)h_{n-1}(\lambda)],
\end{eqnarray*}
where we set $h_0(\lambda)\equiv 1$ and $h_{-1}(\lambda)\equiv 0$.
In particular,
the functions $h_n$ and $-h_n'$ are
positive and decreasing on $(0,1)$ for each $n \in \mathbb N.$

\begin{lemma}\label{chi}
For all $\beta\in (0,\pi/2)$ and $R>0$  there exists
$b=b(\beta,R) \in  (0,\min\{1,1/(2R)\})$ such that for every $n
\in \mathbb N,$
\[
|{\rm Im}\,h_n(re^{i\beta})|\le -\frac{\pi}{2}r h_n'(b r), \qquad r\in (0,R).
\]
\end{lemma}

\begin{proof}
Let  $\beta\in (0,\pi/2)$ and $R>0$ be fixed. For every $r \in
(0,R),$ \begin{eqnarray*}
{\rm Im}\,h_n(re^{i\beta})=
\frac{h_n(re^{i\beta})-h_n(re^{-i\beta})}{2i}
&=&\frac{1}{2i}\int_{-\beta}^\beta \frac{d h_n(re^{i\gamma})}{d\gamma}\,d\gamma\\
&=&-n r\int_{-\beta}^\beta \frac{h_{n-1}(re^{i\gamma})}{(1+re^{i\gamma})^2}\,e^{i\gamma}\,d\gamma.
\end{eqnarray*}
Let $b_\gamma=b(\gamma, R)$ and $\gamma \in (0,\beta]$ be given by Lemma \ref{R}
(see (\ref{alpha})).
Then, using Lemma \ref{R} and the monotonicity
of  $-h_n'$ on $(0,1),$
we obtain for each $r\in (0,R):$
\begin{eqnarray*}
|{\rm Im}\,h_n(re^{i\beta})|&\le&
2\beta n r\sup_{\gamma\in (0,\beta)}\frac{|h_{n-1}(re^{i\gamma})|}{|1+re^{i\gamma}|^2}\\
&\le& 2 \beta n r\sup_{\gamma\in (0,\beta)}\,
\frac{h_{n-1}(b_\gamma r)}{(1+b_\gamma r)^2}\\
&\le& 2 \beta n r\,
\frac{h_{n-1}(b_\beta r)}{(1+b_\beta r)^2}\\
&=&-\beta r h_n'(b_\beta r)\\
&\le& -\frac{\pi}{2}r h_n'(b_\beta r).
\end{eqnarray*}
\end{proof}

Now Lemma \ref{CorD} follows directly from  Lemma \ref{chi}.
Indeed, if
\[
c_n\ge 0, \,\, n\ge 0,  \qquad \sum_{n=0}^\infty c_n=1,
\]
and
\[
\mathbf{h}(\lambda):=1-\sum_{n=0}^\infty c_n h_n(\lambda)=\sum_{n=0}^\infty c_n q_n(\lambda),\qquad \lambda\in \C_+,
\]
then, by Lemma \ref{chi}, for all $\beta\in (0,\pi/2)$ and $R>0,$
there exists $$b=\cos \beta/(1+R^2) \in \left(0,\min\{1,1/(2R)\}\right)$$ such
that
\[
|{\rm Im}\, \mathbf{h} (re^{i\beta})|\le \frac{\pi}{2}r \mathbf{h}'(b r), \qquad r \in (0,R).
\]
In other words, $\mathbf{h} \in \mathcal{D}_{\pi/2}(0,1)$ with
$m=\pi/2$ and $b$ as above.

Finally, if $\lambda \in \mathbb C_+,$ then
$$
{\rm Re}\, \mathbf{h}(\lambda)\ge 1-\sum_{k=0}^{\infty} c_k \left|\frac{1-\lambda}{1+\lambda} \right|^k \ge 0,
$$
and, since  $\mathbf{h}((0,\infty)) \subset [0,\infty),$ we conclude that $\mathbf{h} \in \mathcal{NP_+}.$
\section{Appendix B}

In this appendix, we prove several estimates which allowed us to
obtain additional, geometric  information on Hausdorff functions
of Ritt operators in Section \ref{improvesec}. We start with the
proposition needed in Example \ref{zeta}.

\begin{prop}\label{Lfun}
For every $\alpha\in (0,1),$
\begin{equation}\label{LL}
L_{1+\alpha}(1)-L_{1+\alpha}(\lambda)\in
\overline{\Sigma}_{\alpha\pi/2},\qquad {\rm Re}\,\lambda\le 1.
\end{equation}
\end{prop}

\begin{proof}
Recall that by \cite[p.27, 1.11(3)]{Bateman2}, for every
$\alpha\in (0,1),$
\[
L_{1+\alpha}(\lambda)=\frac{\lambda}{\Gamma(1+\alpha)}\int_0^1
\frac{\log^\alpha(1/s)\,ds}{1-s\lambda}, \quad \lambda\in \C_{+}\setminus
(1,\infty).
\]
and by Proposition \ref{prop1} we have
\begin{equation}\label{MM}
\psi(\lambda):=1-L_{1+\alpha}(1-\lambda)=\frac{1}{\Gamma(1+\alpha)}\int_0^\infty
\frac{\lambda\log^\alpha(1+t)\,dt}{(\lambda+t)t},\quad \lambda\in
\C\setminus(-\infty,0].
\end{equation}
If now $\lambda=|\lambda|e^{i\delta}$, $\delta\in (-\pi,\pi)$, then setting
$t=\lambda\tau$ and using Cauchy's theorem, we infer from (\ref{MM})
that
\begin{equation}\label{MM1}
\psi(\lambda)=\int_0^{e^{-i\delta}\infty}
\frac{\log^\alpha(1+\lambda\tau)\,d\tau}{(1+\tau)\tau} =\int_0^\infty
\frac{\log^\alpha(1+\lambda \tau)\,d\tau}{(1+\tau)\tau}.
\end{equation}
If, moreover, $\lambda\in \C_{+}$, then for any $\tau>0$ we have
\[
\log^\alpha(1+\lambda\tau)\in \Sigma_{\alpha\pi/2},\quad \alpha\in
(0,1).
\]
Thus taking into account (\ref{MM1}) we get $\psi(\C_{+})\subset
\overline{\Sigma}_{\alpha\pi/2},$ and thus (\ref{LL}).
\end{proof}

Now we prove an auxiliary result which is crucial in Example
\ref{eps}.
\begin{lemma}\label{QBC}
If
\[
h(\lambda)=\frac{\lambda-1}{\log \lambda}, \qquad \lambda \in \C_+,
\]
then
\[
h(\D_1)\subset \overline{\Sigma}_{\pi/3},\qquad \D_1=\{\lambda\in
\C: |\lambda-1|<1\}.
\]
\end{lemma}

\begin{proof}
Recall that $h \in \mathcal{CBF},$
hence
\begin{equation}\label{PropB}
{\rm Im}\,h(\lambda)>0,\qquad h(\lambda)=\overline{h(\overline {\lambda})},\quad {\rm Im}\,\lambda>0,
\end{equation}

Let us first prove the following claim:
\begin{equation}\label{incl}
h(\lambda)\in \overline{\Sigma}_{\pi/4},\qquad \lambda\in \overline{\D}_{+},\quad
\D_{+}:=\C_{+}\cap \D.
\end{equation}

For every $s>0$ we have
\[
h(is)=\frac{is-1}{\log s+i\pi/2}
=\frac{(\pi s/2-\log s)+i(s\log s+\pi/2)}{(\log s)^2+\pi^2/4}.
\]
If $s \in (0,1),$ then
\[
\frac{{\rm Im}\,h(is)}{{\rm Re}\,h(is)}=\frac{s\log
s+\pi/2}{\pi s/2-\log s}\le 1 \quad \text{and} \quad h(is)\in \overline{\Sigma}_{\pi/4}.
\]
Moreover, for $s\in (0,1),$
\[
\frac{d}{ds}|h(is)|^2=2\frac{s\pi^2/4-s^{-1}\log s}{((\log s)^2
+\pi^2/4)^2}>0.
\]
So, $|h(i\cdot )|$ is a strictly increasing function on $(0,1),$
$|h(0)|=0$, and $|h(i)|=\frac{2\sqrt{2}}{\pi}.$
Next, for every $\beta\in (0,\pi/2),$
\[
\frac{{\rm Im}\,h(e^{i\beta})}{{\rm Re}\,h(e^{i\beta})}=
\frac{1-\cos\beta}{\sin\beta}=\tan (\beta/2)\le 1 \quad \text{and} \quad h(e^{i\beta})\in  \overline{\Sigma}_{\pi/4}.
\]
Moreover, if $\beta \in (0,\pi/2),$
then
\[
\frac{d}{d\beta}|h(e^{i\beta})|
=\frac{\beta\cos(\beta/2)-2\sin(\beta/2)}{\beta^2}<0.
\]
Hence, $|h(e^{i\cdot})|$ is a strictly decreasing function on $(0,\pi/2),$
$|h(0)|=1,$  and $|h(i)|=\frac{2\sqrt{2}}{\pi}<1.$

Now from  (\ref{PropB}) it follows that $h$ maps $\partial
\D_{+}$ into $\partial \overline{\Sigma}_{\pi/4}$ injectively, and
the claim is proved.

Finally, since $h \in \mathcal{CBF},$ we have
\begin{equation}\label{ar}
h(\overline{\Sigma}_{\pi/3} \setminus \{0\})\subset \overline{\Sigma}_{\pi/3},
\end{equation}
by \eqref{PrB}.
Taking into account $ \D_1\subset \overline{\Sigma}_{\pi/3} \cup
\D_{+}, $ and, using (\ref{incl}) and (\ref{ar}),  we obtain
\[
h(\D_1)\subset \overline{\Sigma}_{\pi/3} \cup \overline{\Sigma}_{\pi/4}=\overline{\Sigma}_{\pi/3} .
\]
\end{proof}

\section{Acknowledgements}
We are grateful to D. Seifert for a careful reading of the manuscript. We would also like to thank the referee for his/her
helpful comments and remarks.

\end{document}